\documentclass[reqno,11pt]{amsart}

\usepackage{amsmath, amsthm, a4, latexsym, amssymb}

\def\@rmrk#1#2{\refstepcounter
    {#1}\@ifnextchar[{\@yrmrk{#1}{#2}}{\@xrmrk{#1}{#2}}}

\makeatletter\@addtoreset{equation}{section}\makeatother

 \newfont{\bfit}{cmbxti10 scaled 1200}

 \newcommand{\reals}{{\mathbb{R}}}
 \newcommand{\eps}{\varepsilon}

 \newcommand{\R}{\mathbb{R}}
 \newcommand{\N}{\mathbb{N}}
 
 \newcommand{\prob}{\mathbb{P}}

 \newcommand{\me}{\mathbb{E}}

 \newcommand{\1}{{\sf 1}}

 \newcommand{\skric}{{\mathcal C}}

 \newcommand{\skrik}{{\mathcal K}}
 \newcommand{\skril}{{\mathcal L}}
 \newcommand{\skrim}{{\mathcal M}}
 
 \newcommand{\skrip}{{\mathcal P}}

 \newcommand{\skrit}{{\mathcal T}}

 \newcommand{\skrix}{{\mathcal X}}
 \newcommand{\skriy}{{\mathcal Y}}
 \newcommand{\heap}[2]{\genfrac{}{}{0pt}{}{#1}{#2}}
 \newcommand{\sfrac}[2]{\mbox{$\frac{#1}{#2}$}}

\newenvironment{Proof}[1]
{\vskip0.1cm\noindent{\bf #1}}{\vspace{0.15cm}}
\renewcommand{\subsection}{\secdef \subsct\sbsect}
\newcommand{\subsct}[2][default]{\refstepcounter{subsection}
\vspace{0.15cm}
{\flushleft\bf \arabic{section}.\arabic{subsection}~\bf #1  }
\nopagebreak\nopagebreak}
\newcommand{\sbsect}[1]{\vspace{0.1cm}\noindent
{\bf #1}\vspace{0.1cm}}

{\nopagebreak {\hfill\rule{2mm}{2mm}}\\ }

\newtheorem{theorem}{Theorem}[section]
\newtheorem{lemma}[theorem]{Lemma}
\newtheorem{cor}[theorem]{Corollary}

\newtheoremstyle{thm}{1.5ex}{1.5ex}{\itshape\rmfamily}{}
{\bfseries\rmfamily}{}{2ex}{}

\newtheoremstyle{rem}{1.3ex}{1.3ex}{\rmfamily}{}
{\itshape\rmfamily}{}{1.5ex}{}
\theoremstyle{rem}
\newtheorem{remark}{{\slshape\sffamily Remark}}[]

\refstepcounter{subsubsection}

\def\thebibliography#1{\section*{Bibliography}
  \list%
  {\arabic{enumi}.}%                          *** style of reference number ***
    {\settowidth\labelwidth{[#1]}\leftmargin\labelwidth
    \advance\leftmargin\labelsep
    \parsep0pt\itemsep0pt
    \usecounter{enumi}}
    \def\newblock{\hskip .11em plus .33em minus .07em}
    \sloppy                   % \clubpenalty4000\widowpenalty4000
    \sfcode`\.=1000\relax}

\begin{document}
%%%%%%%%%%%%%%%%%%%%%%%%%%%%%%%%%%%%%%%%%%%%%%%
\title[Large deviations for trees
] {\Large Large deviation Results for Critical Multitype
Galton-Watson trees}

\author[Kwabena Doku-Amponsah]{}

\maketitle \centerline{\sc{By Kwabena Doku-Amponsah}}
\centerline{\textit{University of Ghana}}

\thispagestyle{empty}
\vspace{0.2cm}

%\centerline{\small Draft of \version}

\vspace{0.5cm}

\centerline{\bf \small Abstract} \begin{quote}{\small In  this  article, we prove a joint large deviation principle in
$n$ for the \emph{empirical pair measure} and\emph{ empirical
offspring measure} of critical multitype Galton-Watson trees
conditioned to have exactly $n$ vertices in the weak topology. From
this result we  extend the large deviation principle for the
empirical pair measures of Markov chains on simply generated
trees  to  cover  offspring  laws  which  are not treated by \cite[Theorem~2.1]{DMS03}. For the case where the offspring law of
the tree is a geometric distribution with parameter  $\sfrac{1}{2}$,
we get an exact rate function. All our rate functions are expressed
in terms of relative entropies. }
\end{quote} \vspace{0.5cm}

\begin{tabular}{lp{13cm}}
{\it Keywords:} & Tree-indexed Markov chain, critical Galton-Watson
tree, joint large deviation
principle, empirical pair measure, empirical transition measure, empirical offspring measure, sub-consistency of  empirical measures, weak shift-invariance.\\
\multicolumn{2}{l}
\end{tabular}
\renewcommand{\thefootnote}{1}
\renewcommand{\thefootnote}{}
\footnote{\textit{Address:} Statistics Department, University of
Ghana, Box LG 115, Legon,Ghana.\,
\textit{E-mail:\,kdoku@ug.edu.gh}.}

\section{Introduction~and~Background}\label{intro}

For the  past  decades, conditioned Galton-Watson  trees have
received an increasing attention from researchers, see, e.g.
\cite{Al91a}, \cite{Al91b}, \cite{Al93},\cite{AP98} or \cite{SJ12}
and the references therein.These random trees have proved  to  be  extremely  good  in modelling
 phenomena which occur in natural hierarchy, example
mutations in mitochondrial DNA \cite{OS02}.

Large  deviation studies of critical Galton-Watson trees
conditioned on the total size was first studied by Dembo,
M\"{o}rters and Sheffield ~\cite{DMS03}. In their article, concepts
such  as shift-invariance and specific relative entropy were extended to
Markov fields on random trees. With these concepts, large deviation
principles for empirical measures of a class of random trees
including Galton-Watson trees conditioned to have  exactly
$n$ vertices were proved  in a topology \emph{stronger} than the weak.
The  strong  topology  (which  necessitated the use of  strict moment assumption)
restricted  their  study to only Galton-Watson trees with offspring laws
super-exponentially decay at infinity. i.e. offspring law $p(\cdot)$ with  all its
exponential moments finite.

Their  paper also look at  the  large deviation
principle for the empirical offspring measure of multitype
Galton-Watson trees  but for only offspring laws   with  all %its
exponential moments finite, see \cite{DMS03}.

%However,  their   result failed to  address  offspring laws which
%have sub-exponential and exponential decaying function at  infinity.
%Example, a {\em Markov chain indexed by geometric $ \sfrac 1 2$
%offspring  law}.

The  aim of  this  article is to carry out a non-trivial  extension
of the large deviation principle for  the tree  indexed Markov chain
of \cite[Theorem~2.2]{DMS03} to cover offspring laws not discussed
by the  paper,  for example, a {\em Markov chain indexed by geometric $ \sfrac 1 2$
offspring  law}.

To  be  specific, we prove a joint large deviation principle for the
empirical pair measure and  empirical offspring measure of multitype
Galton-Watson trees having critical offspring laws with finite
second moments. This includes offspring laws considered by
~\cite{DMS03}.

To deal  with the problem of  exponential tightness  in  the strong
topology  encounted in \cite{DMS03} which necessitated the use  of
strong moment condition, we define the
concept of \emph{consistency}   for empirical measures  of  multitype Galton-Watson trees, see,  example  Doku-Amponsah and M\"{o}rters~\cite{DM10}. With this
concept, we prove the upper bound, using similar strategy as  in
\cite{DMS03}  under only finite second moment condition in the weak
topology.

 Our proof of
the lower bound  unlike the  proof  of their lower bound  uses a
truncation argument for vertices  with  too  many  offspring. To  be
specific, for  a given Multitype Galton-Watson tree we obtain
another Multitype Galton-Watson tree  by restricting the offspring
distribution  to some bounded set $\skrix_k^{*}.$ Taking appropriate
limit as   $k$  goes to infinity we obtain the results as  a limit.

Using the contraction principle, see Dembo \cite{DZ98}, we derive
from our main results large deviation principle for the empirical pair
measure of Markov chain indexed by random trees. This result is
similar to the one in \cite{DMS03}. We remark here that the process
level large deviation principles for the empirical subtree measure
and single-generation empirical measure, see \cite{DMS03}, can be
developed from our main results.

Specifically, we consider random tree models where trees and types
are chosen \emph{simultaneously} according to a \emph{multitype
Galton-Watson tree}. We recall from \cite{DMS03}  the model of
multitype Galton-Watson tree. Let  $\skriy$ be  a  finite  alphabet. Write $\skriy^*=\bigcup_{n=0}^\infty \{n\} \times \skriy^n$
 and endow  it with the discrete topology. Denote by
$\skrit$ the set of all finite rooted planar trees $T$, by $V=V(T)$
the set of all vertices and by $E=E(T)$ the set of all edges
oriented away from the root, which is always denoted by $\sigma$. Write $|V(T)|$ for the number of vertices in the tree $T$. Note that
the offspring of any vertex $v\in T$ may  be  represented by an element
of $\skriy^*.$ Let the element $(0,\emptyset)\in\skriy^*$ denote  lack of offspring. Associated with  every typed tree $Y$ and each vertex $v$ we denote by
$(N(v),Y_1(v),\ldots,Y_{N(v)}(v))\in \skriy^*$ the
number and types of the children of $v$, ordered from left to right.

Let $\mu:\skriy\to[0,\,1]$  be the initial
law, and  $\skrik$ be the offspring transition kernel. Define  a
tree-indexed process $Y$, see Pemantle~\cite{Pe95}, as follows:
\begin{itemize}
\item Assign the root $\sigma$  a random type $Y(\sigma)$
chosen according to the law $\mu$ on $\skriy$.
\item Give independently of everything else, each vertex with type $a\in \skriy$ the offspring number
and types, by
offspring law $\skrik\{ \, \cdot\, \,|\, a\}$ on $\skriy^*$. We write
$$\skrik\{ \, \cdot\, \,|\, a\}
=\skrik\{ (N,Y_1,\ldots, Y_N)\in \cdot \,\, | \, a\},$$ i.e. we have a
random number $N$ of offspring particles with types $Y_1,\dots,Y_N$.
\end{itemize}
Let $c=(n,a_1,\ldots,a_n)\in \skriy^*$ and $a\in\skriy.$  Define  the
\emph{multiplicity} of the symbol~$a$ in $c$  by
$\displaystyle \ell(a,c)=\sum_{i=1}^n \1_{\{a_i=a\}}$ and  the matrix
$M$ with  nonnegative entries by

$$\displaystyle M(a,b)=\sum_{c\in \skriy^*} \skrik\{ c \, | \, b\}
\ell(a,c), \mbox{ for } a,b\in\skriy.$$  i.e.~$M(a,b)$ are the expected
number of offspring of type~$a$ of a vertex of type~$b$.  We also
recall from \cite{DMS03} the weak form of irreducibility concept.
 With $M_*(a,b) =
\sum_{k=1}^\infty M^k(a,b) \in [0,\infty]$ we say that the matrix
$M$ is \emph{weakly irreducible} if $\skriy$ can be partitioned into
a non empty set $\skriy_r$ of \emph{recurrent states} and a disjoint
set $\skriy_t$ of \emph{transient states} such that
\begin{itemize}
\item $M_*(a,b)>0$ whenever $b \in \skriy_r$, while
\item $M_*(a,b)=0$ whenever $b \in \skriy_t$ and either $a=b$ or $a \in \skriy_r$.
\end{itemize}
For example, any \emph{irreducible} matrix $M$ has $M_*$ strictly
positive, hence is also weakly irreducible with $\skriy_r=\skriy$.
The multitype Galton-Watson tree is called weakly irreducible (or
irreducible) if the matrix $M$ is weakly irreducible (or
irreducible, respectively) and the number $\sum_{a \in \skriy_t}
\ell(a,c)$ of transient offspring is uniformly bounded under $\skrik$.

Recall that, by the Perron-Frobenius theorem, see e.g. \cite[Theorem
3.1.1]{DZ98}, the largest eigenvalue of an irreducible matrix is
real and positive. Obviously, the same applies to weakly irreducible
matrices.  The multitype Galton-Watson tree is called
\emph{critical} if this eigenvalue is $1$ for the matrix $M$.\\

The remaining part of the article is organized in the following
manner: The complete statement of our results is given in Section
\ref{statement}, we  begin with joint LDP for empirical pair
measures and empirical offspring measures of multitype Galton-Watson
trees, followed by a corollary of the LDP for the empirical offspring
measure of multitype Galton-Watson trees in
subsection~\ref{offspring}. In subsection~\ref{pair}, we state the
LDP for empirical pair measures of Markov chains indexed by a tree.
The proofs of our main results are then given in Section
\ref{proofs}. All corollaries and Theorem~\ref{main} are proved in
Section~\ref{cormain}.\\

The  main  difference between  this  paper and \cite{DMS03} lies in
the  the  topologies  used.  The  rate function  in
\cite[Theorem~2.2]{DMS03} is not  continuous  in  the weak topology
but it is  continuous in a topology  stronger enough to make the
functionals $\pi\mapsto \int f(b,c)\, \pi(db\, ,dc)$ continuous, for
$f:\skriy\times\skriy^*\to\R$ either bounded, or
$f(b,c)=\ell(a,c)\1_{b_0}(b)$ for some $a,b_0 \in\skriy$. In  this
strong  topology,   sequence  of  empirical  offspring measures of
Multi-type   Galton-Watson  tree is exponentially  tight  only if
the  offspring distribution   of   the tree has all its exponential
moments  finite. We  note here that, our rate functions are
continuous  in  the  weak  topology,  and  in the weak topology only
some  moments  are  necessary for establishing exponential
tightness.

\section{Statement of the results} \label{statement}

\subsection{ Joint large deviation principle for empirical pair measure and
empirical offspring measure of critical multitype Galton-Watson
trees}\label{offspring}

We  assume  throughout the  remaining part of  this  paper  that
$T$  is  a  finite  tree.

For every sample chain $Y$, we associate the \emph{empirical
children measure} $\skrim_Y$ on $\skriy\times\skriy^*$, by
\begin{equation}\label{Def.MX}
\skrim_Y(a,c)=\frac 1{|V(T)|} \sum_{v\in V} \delta_{\big (X(v),\,(N(v), Y_1(v),Y_2(v),\,...,Y_N(v))\big )} (a,c),\,\, \mbox{$a\in \skriy$,  $c\in\skriy^*$ }
\end{equation}
and the empirical pair measure on  $\skriy\times\skriy$,  by
\begin{equation}\label{Def.LX} \tilde{\skril}_Y(a,b)= \frac{1}{|V(T)|} \sum_{\sigma\in
E} \delta_{(Y(\eta_1), Y(\eta_2))}(a,b), \,\, \mbox{ for } a,b\in\skriy,
\end{equation}
where $\eta_1, \eta_2$ are the beginning and end vertex of the edge $\eta\in
E$ (so $\eta_1$ is closer to $\eta$ than $\eta_2$).
 We note that $$\tilde{\skril}_Y(a,b)=\sum_{c\in\skriy^*}\ell(b,c)\skrim_Y(a,c).$$
By definition, we notice that $\skrim_Y$  is a probability vector and
that total mass $\|\tilde{\skril}_Y\|$ of $\tilde{\skril}_Y$ is
$\sfrac{|V(T)|-1}{|V(T)|}\le 1.$

Our main result is an LDP for
$(\tilde{\skril}_Y,\,\skrim_Y)$ if $Y$ is a critical multitype Galton-Watson
tree.

We denote by $\skrip(\skriy\times\skriy^*)$ the space of probability
measures $\pi$ on $\skriy\times\skriy^*$ with $\int n \, \pi(da \, ,
dc)<\infty$, using the convention $c=(n,a_1,\ldots,a_n)$.   Denote by
$\tilde{\skrip}(\skriy\times\skriy)$ the space of finite measures $\omega$  on
$\skriy\times\skriy$  with  $\|\omega\|\le 1$ and  endow the space
$\tilde{\skrip}(\skriy\times\skriy)\times\skrip(\skriy\times\skriy^*)$
with the weak topology.  We also  endow
$\Big\{(\omega,\pi)\in\tilde{\skrip}(\skriy\times\skriy)\times\skrip(\skriy\times\skriy^*):
\omega_2=\pi_1\Big\}$  with  the  weak  topology.

We call $(\omega,\,\pi)\in\tilde{\skrip}(\skriy\times\skriy)\times\skrip(\skriy\times\skriy^*)$
\emph{sub-consistent}  with  respect to the  weak topology if
\begin{equation}\label{consistent}
\omega(a,b)\ge \sum_{c\in \skriy^*} \ell(b,c) \pi(a,c),\, \mbox{ for
all $a\in\skriy$ and  $b\in\skriy$.}
\end{equation}
%or

%\begin{equation}\label{consistentesaimp}
%\omega(a,b)= \sum_{c\in \skriy^*} \ell(b,c) \pi(a,c),\, \mbox{ for all
%$a\in\skriy$ and $b\in\skriy$ (with  all  $c\to \ell(b,c)$ bounded).}
%\end{equation}

 It is called \emph{consistent} if equality hold in
\eqref{consistent}.Observe that, if $(\omega,\,\pi)$ is empirical
pair measure and empirical offspring measure of a multitype
Galton-Watson tree then  \eqref{consistent} is
$$\sfrac{1}{n}\times\sharp\big\{\mbox{edges with beginning vertex
of type $a$ and end vertex of type $b$}\big\}.$$

Note  that by  definition any consistent $(\omega,\,\pi)$  is
sub-consistent but not all sub-consistent elements of
$\tilde{\skrip}(\skriy\times\skriy)\times\skrip(\skriy\times\skriy^*)$
are  consistent.

We call an offspring distribution $\skrik$ \emph{bounded} if
for some $\,k<\infty,$ we have \,$$\skrik\{N>k\,|\,a\}=0, \,
\mbox{ for all $\,a\in\skriy\,$}.$$ Otherwise we call it
\emph{unbounded}.  We  say the  offspring law $\skrik$  has  a
finite second  moment if
$$\sum_{c\in\skriy^{*}}n(c)^2\,\skrik\{c\,|\,a\}<\infty,\mbox{ for  all $a\in\skriy$}. $$

 To formulate our first LDP, denote by $\pi_1$ the
$\skriy$-marginal of  probability measure $\pi$ on
$\skriy\times\skriy^*$,  by  $\omega_2$  the second marginal of
finite measure $\omega$ on $\skriy\times\skriy$ and write
$$\displaystyle \pi_1\otimes\skrik(a,c):=\pi_1(a)\skrik\{c\,|\,a\}.$$
Further, recall  that the  relative  entropy  of  the  probability
measure $\pi$  with respect $\hat{\pi}$ is given by
$$\displaystyle H(\pi \,\|\,
\hat{\pi})=\sum_{(a,c)\in\skriy\times\skriy^{*}}\pi(a,c)\log\Big[
\sfrac{\pi(a,c)}{\hat{\pi}(a,c)}\Big].$$

\begin{theorem}\label{general}
Suppose that $Y$ is a weakly irreducible, critical multitype
Galton-Watson tree with  offspring law  $\skrik$ whose second moment is
finite, conditioned to have exactly $n$ vertices. Then, for
$n\to\infty$, the  pair $(\tilde{\skril}_Y,\,\skrim_Y)$ satisfies an LDP in $\tilde{\skrip}(\skriy\times\skriy)\times\skrip(\skriy\times\skriy^*)$ with speed $n$ and the
convex, good rate function
\begin{equation}\label{equ-rate12}
J(\omega,\,\pi)=\left\{ \begin{array}{ll} H(\pi \,\|\,
\pi_1\otimes\skrik) &
\mbox{ if $(\omega,\,\pi)$ is sub-consistent and  $\omega_2=\pi_1$ }\\
\infty & \mbox{ otherwise.}
\end{array} \right.\end{equation}
%where  $\omega_2$ is the second marginal of the finite measure
%$\omega$ and $\pi_1$ is the $\skriy-$ marginal of  the probability
%measure $\pi.$
\end{theorem}

\begin{remark}
Note that,  the  functional  relationship
$\tilde{\skril}_Y(a,b)=\sum_{c\in \skriy^*} \ell(b,c) \skrim_Y(a,c)$  may  break
down in  the  limit,  because  $\pi\rightarrow\sum_{c\in \skriy^*}
\ell(\cdot,c) \pi(\cdot,c)$  is  discontinuous  in  the  weak topology.
The possibility  of  this  effect  is  responsible  for the  weak
form of  the  condition  $(\omega,\,\pi)$  sub-consistent in  the
rate function  \eqref{equ-rate12}.
\end{remark}

\begin{remark}
  Observe here  that  the  erratum  on \cite[Theorem~2.1]{DM10}  doses  not  apply  to  Theorem~\ref{general}   above.  In  fact,  \cite[Theorem~2.1]{DM10}  was  proved  by  conditioning  on  the  set  of  consistent  measures  and  therefore,  it  is  required  that the rate  function in \cite[Theorem~2.1]{DM10} is finite on   only consistent measures, see \cite{DM10} and  Erratum. However, the  proof of  Theorem~\ref{general} given in  Section~\ref{proofs}  is not  by  conditioning  on consistent  measures   and  so,  the rate  function \eqref{equ-rate12}  is  finite not only on  consistent measure but  all  sub-consistent measures.
\end{remark}

From  Theorem~\ref{general} and  the  contraction  principle,  see
\cite[Theorem~4.2.1]{DZ98},   we  obtain  a  large  deviation
principle  for  the  empirical  offspring measure $\skrim_Y$ in  the weak
topology.  To  state  this LDP,  we  call
$\pi\in\skrip(\skriy\times\skriy)$ \emph{weak shift-invariant} (with
respect to the  weak topology) if

\begin{equation}\label{weakshift}
\pi_1(b)\ge\sum_{(a,c)\in\skriy\times\skriy^*} \ell(b,c) \pi(a,c),\,
\mbox{ for all  $b\in\skriy$.}
\end{equation}

It is  called  \emph{ shift-invariant } if equality hold in
\eqref{weakshift}. Note, that  if $\pi$ is empirical offspring
measure of a multitype Galton-Watson tree then \eqref{weakshift} is
$$\sfrac{1}{n}\times\sharp\big\{\mbox{ vertices of type $b$}\big\},$$
and  so, \eqref{weakshift}  above  coincide with  the definition  of
shift-invariant  by   \cite[(2.9)]{DMS03}.

\begin{remark}
If $\skrip(\skriy\times\skriy)$  is  equipped with  the  stronger
topology  of  \cite{DMS03} then  every  weak shift-invariant measure
in $\skrip(\skriy\times\skriy)$   is  shift-invariant.  Otherwise,
if  $\pi\in \skrip(\skriy\times\skriy)$ is weak shift-invariant  and
there  exists  $b_0\in\skriy$  such  that  we have
$\pi_1(b_0)>\sum_{(a,c)\in\skriy\times\skriy^*} \ell(b_0,c) \pi(a,c),$
then we  can  find consistent  $(\omega_n,\pi_n),$ see \cite[Lemma~4.6]{DM10},  with
$(\pi_n)_1=(\omega_n)_2,$ $\pi_n$ converging to $\pi$ which will
then  lead to the contradiction
$$\begin{aligned}\pi_1(b_0)>\sum_{(a,c)\in\skriy\times\skriy^*} \ell(b_0,c)
\pi(a,c)&=\sum_{(a,c)\in\skriy\times\skriy^*}
\ell(b_0,c)\lim_{n\to\infty} \pi_n(a,c)=\lim_{n\to\infty}
\sum_{a\in\skriy}\sum_{c\in\skriy^*}
\ell(b_0,c)\pi_n(a,c)\\
&=\lim_{n\to\infty}
\sum_{a\in\skriy}\omega_n(a,b_0)=\lim_{n\to\infty}(\omega_n)_2(b_0)=\lim_{n\to\infty}(\pi_n)_1(b_0)=\pi_1(b_0).
\end{aligned}$$
\end{remark}

\begin{cor}\label{general-cor}
Let $Y$ be a weakly irreducible, critical multitype
Galton-Watson tree with an  offspring law  $\skrik$ whose second moment
is finite, conditioned to have exactly $n$ vertices. Then, for
$n\to\infty$, the empirical offspring measure~$\skrim_Y$ satisfies an LDP in $\skrip(\skriy\times\skriy^*)$ with
speed $n$ and the convex, good rate function
\begin{equation}\label{equ-ratecor}
\Phi(\pi)=\left\{ \begin{array}{ll} H(\pi \,\|\, \pi_1\otimes\skrik) & \,
\mbox{ if \,$\pi$ is  weak shift-invariant,}\\
\infty & \mbox{ otherwise.}
\end{array} \right.\end{equation}
\end{cor}
Here, we remark that \emph{finite second moment} assumption in
Theorem~\ref{general} and Corollary~\ref{general-cor} is necessary
for us to establish the sub-exponential decay of the probability of
the event $\{|V(T)|=n\}$ on the set
$$\big\{n\in\N:\,\prob\{|V(T)|=n\}>0\big\}.$$ See
\cite[Lemma~3.1]{DMS03}. We write $\skriy_k^*=\bigcup_{n=0}^k \{n\}
\times \skriy^n$ and  notice  that $\skriy_k^*\subset\skriy^*,$  for
all $k\in\N.$

And  by  $\skrik_k\big\{
 \cdot\,|\,a\big\}$   we  denote  an  offspring transition  kernel  with
 support $\skriy_k^*.$ Thus,  we  have   $$\skrik_k\big\{
 \skriy_k^*\,|\,a\big\}=\sum_{c\in\skriy_k^*}\skrik_k\{c\,|a\,\}=1,\, \mbox{ for all  $a\in\skriy$}.$$

 The next large  deviation principle is the main ingredient in the
proof of the lower bound of Theorem~\ref{general}.

\begin{theorem}\label{general2}
Let $Y$ be  a weakly irreducible, critical multitype
Galton-Watson tree with an offspring law  $\skrik_k$,conditioned to have
exactly $n$ vertices. Then, for $n\to\infty$, $(\tilde{\skril}_Y,\,\skrim_Y)$
satisfies a large deviation principle in
$\tilde{\skrip}(\skriy\times\skriy)\times\skrip(\skriy\times\skriy_k^*)$ with speed $n$ and the convex, good rate function

\begin{equation}\label{equ-rate2}
J_k(\omega,\,\pi)=\left\{ \begin{array}{ll} H(\pi \,\|\,
\pi_1\otimes\skrik_k) &
\mbox{ if $(\omega,\,\pi)$ is consistent and $\omega_2=\pi_1$ }\\
\infty & \mbox{ otherwise.}
\end{array} \right.\end{equation}

\end{theorem}
\begin{remark}

In  the  space
$\tilde{\skrip}(\skriy\times\skriy)\times\skrip(\skriy\times\skriy_k^*)$
every  sub-consistent  pair of  measures  is consistent. Otherwise
if $(\omega,\,\pi)$  sub-consistent  and  there  exits  some
$a_0,b_0\in\skriy$  such  that $\omega(a_0,b_0)>\sum_{c\in \skriy^*}
\ell(b_0,c) \pi(a_0,c),$  then  we can  construct a consistent
$(\omega_n,\,\pi_n),$  see \cite[Lemma~4.6]{DM10}, converging to
$(\omega,\,\pi)$ which will lead  to the  contradiction
$$\begin{aligned}\omega(a_0,b_0)>\sum_{c\in \skriy_k^*} \ell(b_0,c)\pi(a_0,c)=\sum_{c\in \skriy_k^*} \ell(b_0,c)
\lim_{n\to\infty}\pi_n(a_0,c)&=\lim_{n\to\infty} \sum_{c\in
\skriy_k^*} \ell(b_0,c)
\pi_n(a_0,c)\\
&=\lim_{n\to\infty}\omega_n(a_0,b_0)=\omega(a_0,b_0).
\end{aligned}$$

\end{remark}

\begin{remark}
Theorem~\ref{general}  and Theorem~\ref{general2} give  the same
large  deviation  principle  with  good rate  function
$J_k(\omega,\,\pi)=J(\omega,\,\pi)$ when  $\skrik=\skrik_k$  for  some  $k.$

\end{remark}

\subsection{LDP for empirical pair measure of Markov
chains indexed by trees}\label{pair}

In this subsection, we look at
the situation where the tree is generated independently of the
types.%,in which case our rate functions are particularly simple.

Let $T$ be any finite tree, $\mu$ a probability measure on a finite alphabet $\skriy$ and
$K$ a Markovian transition kernel. A \emph{ Markov chain indexed  by tree} $Y:V \to
\skriy$ may  be  obtained as follows: Choose $Y(\sigma)$ according to $\mu$ and choose
$Y(v)$, for each vertex $v\not=\sigma$, independently of everything else, according to  the transition kernel
given the value of its parent. If
the tree is  randomly chosen, we shall look at  $Y=\{Y(v)\, : \,
v\in V(T)\}$ under the \emph{joint law} of tree and chain. It is
sometimes convenient to take $Y$ as a \emph{typed tree} and  consider  $Y(v)$ as the
\emph{type} of the vertex $v$.\\

We consider the class of \emph{simply generated trees}, see
\cite{MM78} or \cite{Al91a},  obtained by conditioning a critical
Galton-Watson on its total number of vertices. To be specific, we
look at the class of Galton-Watson trees, where the number of
children $N(v)$ of each $v \in V(T)$ is chosen independently  according
to the same law $p(\,\cdot\,)=\prob\{N(v)=\,\cdot\,\}$ for all $v
\in V(T)$, while $0<p(0)<1$. We assume that $p$ is critical.  That is,
the mean offspring number $\sum_{n=0}^\infty n p(n)$ is
one, but this assumption may  be  relaxed for  some  noncritical
cases. %: Note that the distribution of $T$ conditioned on $\{ |V(T)|
%= n \}$ is exactly the same as when the offspring law is
%$p_\theta(\ell) = p(\ell) e^{\theta \ell}/\sum_j p(j) e^{\theta j}$,
%regardless of the value of $\theta \in \reals$. With $0<p(0)<1-p(1)$
%there exists a unique $\theta_*$ such that $\sum_\ell \ell
%p_{\theta_*} (\ell) =1$. Hence all our results hold in the
%noncritical cases with $p_{\theta_*}$ in place of $p$.

We allow offspring laws $p$ with unbounded support, but we relax the
assumption that all exponential moments of  $p$  are finite. In  fact, we replace  the  stronger condition   $n^{-1} \log
p(n) \to -\infty$ of \cite[Theorem~2.1]{DMS03}  with  a weaker
condition $\sum_{n=0}^{\infty}n^{2} p(n)<\infty,$ and obtain an LDP  for  the  empirical pair measure of tree
indexed Markov  chains.  Assume  hereafter that the statement conditioned on the
event $\{|V(T)|=n\}$ are made only for those values of $n$ where the
event $\{|V(T)|=n\}$ has positive probability.

 For each typed tree $Y,$  we recall from \cite{DMS03}, the definition
 of the
\emph{empirical pair (probability) measure}~$\skril_Y$ on
$\skriy\times\skriy$ as
\begin{equation}\label{Def.LX}
\skril_Y(a,b)= \frac{1}{|E|} \sum_{\eta\in E} \delta_{(Y(\eta_1),
Y(\eta_2))}(a,b), \, \mbox{ for } a,b\in\skriy,
\end{equation}
where $\eta_1, \eta_2$ are the beginning and end vertex of the edge $\eta\in
E$ (so $\eta_1$ is closer to $\sigma$ than $\eta_2$). Notice,
$\skril_Y=\sfrac{n}{n-1}\tilde{\skril}_Y$ on the set $\{|V(T)|=n\}$ and hence
the LDP for $\tilde{\skril}_Y$ implies $\skril_Y$ by exponential equivalent
Theorem, see \cite[Theorem~4.2.13]{DZ98}.  Note,  for all
$a\in\skriy,$ the  \emph{empirical transition  measure}
$$\sfrac{\skril_Y(a,\,\cdot)}{\sum_{b\in\skriy}\skril_Y(a,b)}$$
is a statistics for  the Markovian transition kernel
$K\{\cdot\,|\,a).$ For given  empirical pair  measure  $\mu,$ we
write
$$\rho_1(a):=\sum_{b\in\skriy}\rho(a,b)\,\,
, \,\,\rho_2(a):=\sum_{a\in\skriy}\rho(a,b)\,\, \mbox{and}\,\,
\rho(\cdot\,|\,a):=\sfrac{\rho(a,\,\cdot)}{\rho_1(a)},\,
$$
 whenever $\rho_1(a)>0$.  Recall,  for  all  $a\in\skriy,$   the relative entropy  of an
empirical transition  kernel $\rho(\cdot\,|\,a)$ with respect  to a
Markovian transition kernel $K\{\cdot\,|\,a\}$ as

$$H\big(\rho(\cdot\,|\,a) \, \| \,
K\{\cdot\,|\,a\}\big)=\sum_{a\in\skriy}\rho(b\,|\,a)\log \Big
[\frac{\rho(b\,|\,a)}{K\{b\,|\,a\}}\Big].$$  See, e.g.~
\cite[Theorem~3.1.13]{DZ98}.  Our first result in this subsection,
the  LDP for $\skril_Y,$ is an extension of \cite[Theorem~2.1]{DMS03}.

%For its formulation recall the definition of the relative
%entropy $H(\cdot\, \| \, \cdot)$ from \cite[(2.1.5)]{DZ98} and the
%Cram\'er's rate function, see e.g. \cite[(2.1.26)]{DZ98},
%\begin{equation}\label{Ip-def} I_p (x) = \sup_{\lambda \in \reals}
%\, \Big\{ \lambda x -
% \log \Big[ \sum_{n=0}^\infty  p(n) e^{\lambda n} \Big] \Big\} \,.
%\end{equation}

\begin{theorem}\label{main}
Let $T$ be a Galton-Watson tree, with offspring law
$p(\cdot)$ such that $0<p(0)<1-p(1)$, $\sum_{n} n p(n) = 1$
and $\sum_{n}n^{2} p(n)<\infty.$ Suppose that  $Y$ is a Markov chain
indexed by $T$ with arbitrary initial distribution and an
irreducible Markovian transition kernel $K$. Then, for $n\to\infty$,
the empirical pair measure $\skril_Y$, conditioned on $\{|V(T)|=n\}$
satisfies an LDP in $\skrip(\skriy \times
\skriy)$ with speed $n$ and the convex, good rate function

\begin{align}\phi(\rho)=\left\{
\begin{array}{ll} \displaystyle \sum_{a \in \skriy} \rho_1(a) \Big[ H\big(\rho(\cdot\,|\,a) \, \| \, K\{\cdot\,|\,a\}\big)\Big] +
 \, \displaystyle \sum_{a \in \skriy} \rho_2(a)\Big[\phi_p \big(\sfrac{\rho_1(a)}{\rho_2(a)}\big)\Big]
 & \mbox{ if $\rho_1\ll \rho_2$,} \\
\infty & \mbox{ otherwise,}
\end{array} \right.
\label{Idef}
\end{align}
where
\begin{equation}\label{Ip-def}
\phi_p (x) = \sup_{\lambda \in \reals} \, \Big\{ \lambda x -
 \log \Big[ \sum_{n=0}^\infty  p(n) e^{\lambda n} \Big] \Big\} \,.
\end{equation}

\end{theorem}

 $\phi(\rho)$  can  be  interpreted  as   the  cost  of  obtaining the
 empirical pair measure  $\rho,$  this  cost consists of  two sub-costs:
\begin{itemize}
\item[(i)]  $\displaystyle \sum_{a \in \skriy} \rho_1(a) \Big[ H\big(\rho(\cdot\,|\,a) \, \| \,
Q\{\cdot\,|\,a\}\big)\Big]$ represents  the  expected cost  of
obtaining the  empirical  transition  kernel  $\rho(\cdot\,|\,a),$
this   cost  is non-negative and  vanishes iff
$\rho(\cdot\,|\,a)=K\{\cdot\,|\,a\}.$
\item[(ii)] $\displaystyle \sum_{a \in \skriy} \rho_2(a)\Big[\phi_p
\big(\sfrac{\rho_1(a)}{\rho_2(a)}\big)\Big]$  represents  the expected
cost of obtaining an  untypical  indexed  tree  for the  Markov
Chain with empirical  transition  measure  $\rho(\cdot\,|\,a),$ this
cost is non-negative  and  vanishes  iff  $\rho_1=\rho_2.$ i.e. if
$\rho$ is shift-invariant.  See,  \cite[Theorem~3.1.13]{DZ98}.
\end{itemize}
Hence,  $\phi(\rho)$ is  non-negative  and  vanishes  iff
$\rho(\cdot\,|\,a)=K\{\cdot\,|\,a\}$  and  $\rho_1=\rho_2.$

Note that  there  is no  qualitative  difference  between
Theorem~\ref{main} and  \cite[Theorem~2.1]{DMS03}\emph{except}  the
topologies.  Moreover, our result allows  for  the  study  of   more
general offspring distributions, namely offspring laws with (only)
finite second moments. For  instance, our result allows the
formulation of an LDP for the empirical pair measure of critical
Galton-Watson trees with geometric distribution with parameter
$\sfrac 12$  as follows:
\begin{cor}\label{geometric}
Suppose that $T$ is a Galton-Watson tree, with offspring law
$p(n)=2^{-(n+1)},$ $n=0,1,\ldots,$.  Let $Y$ be a Markov
chain indexed by $T$ with arbitrary initial distribution and an
irreducible Markovian transition kernel $K$. Then, for $n\to\infty$,
the empirical pair measure $\skril_Y$, conditioned on $\{|V(T)|=n\}$
satisfies a large deviation principle in %the space of probability vectors on
$\skrip(\skriy \times \skriy)$ with speed $n$ and the convex, good
rate function

\begin{align}
\phi(\rho)=\left\{ \begin{array}{ll} H(\rho \, \| \,\rho_1\otimes K) +
H(\rho_1\,\|(\rho_1+\rho_2)/2)+H(\rho_2\,\|(\rho_1+\rho_2)/2)
 & \mbox{ if $\rho_1\ll \rho_2$,} \\
\infty & \mbox{ otherwise.}
\end{array} \right.
\label{Idef}
\end{align}

\end{cor}
\pagebreak

Consider the  following  example  from  the  field  of biology.

{\bf  Mutations  in  mitochondrial  DNA.} Mitochondria  are
organelles  in  cells  carrying  their  own  DNA.  Like  nuclear
DNA,  mtDNA  is  subject  to  mutations  which  may  take   the form
of  base  substitutions,  duplication  or  deletions.  The
population mtDNA  is  modelled  by  two-type  process  where  the
units  are  $1$  (normals)  and  $0$ (mutant),  and  the  links are
mother-child  relations.  A  normal  can  give  birth  to  either
all normals  or,  if   there  is  mutation,   normals  and mutants.
Suppose  the  latter  happens  with probability or  mutation  rate
$\alpha.$  Mutants can only give birth to mutants. A DNA molecule
may also die without reproducing. We  denote  by  $\emptyset$ the
 event absence of offspring.  Let the survival probabilities be
$p\in\big[0,\sfrac{1}{(2-\alpha)}\big]$ and
$q\in\big[0,\,\sfrac{1}{2}\big]$ for  normals  and  mutants
respectively. Assume that the population is started from one normal
ancestor. Suppose the offspring kernel $\skrik$ is given by
$$\begin{aligned}
&\skrik\big\{(n,
a_1,a_2,a_3,...,a_n)\,|\,1\big\}=\Big(\frac{1}{2}\Big)^{n+1}\prod_{k=1}^{n}K_{\alpha}\{a_k\,|\,1\},\\
&\skrik\big\{(n,
a_1,a_2,a_3,...,a_n)\,|\,0\big\}=\Big(\frac{1}{2}\Big)^{n+1}\prod_{k=1}^{n}K_{\alpha}\{a_k\,|\,0\},\\
\end{aligned}$$

where $K_{\alpha}\{\emptyset\,|\,1\}=1-p,$
$K_{\alpha}\{0\,|\,1\}=p\alpha,$
$K_{\alpha}\{1\,|\,1\}=p(1-\alpha),$ $K_{\alpha}\{0\,|\,0\}=q$  and
$K_{\alpha}\{\emptyset\,|\,0\}=1-q.$

Then, in  the framework  of  Corollary~\ref{geometric}  we  have
$p(n)=2^{-(n+1)},$   and  $\skrik\{b\,|\,a\}=K_{\alpha}\{b\,|\,a\}.$
Note,
 the  law  of   the  tree  $T$  conditional  on  event  $\big\{ |V(T)|=n\big\}$ corresponds to sampling  a  tree  $T$
 uniformly  from all  unordered  trees  with  vertices  $n.$ And given  $T,$ the  process  $Y$   on  the  vertices  form  a  Markov  chain.
 Hence, the  empirical pair  measure  $\skril_Y$  obeys  a  large  deviation  principle, by Corollary~\ref{geometric}  with  the good  rate
    function  \eqref{Idef}   given  by
    $$\begin{aligned}
 \phi(\rho)=\rho(1,0)\log\Big[&\sfrac{\rho(1,0)}{p\alpha\rho_1(1)}\Big]+\rho(1,1)\log\Big[\sfrac{\rho(1,1)}{p(1-\alpha)\rho_1(1)}\Big]
    +\rho(0,0)\log\Big[\sfrac{\rho(0,0)}{q\rho_1(0)}\Big]+\rho_1(1)\log\Big[\sfrac{2\rho_1(1)}{(\rho_1(1)+\rho_2(1))}\Big]\\
    &+\rho_1(0)\log\Big[\sfrac{2\rho_1(0)}{(\rho_1(0)+\rho_2(0))}\Big]
    +\rho_2(1)\log\Big[\sfrac{2\rho_2(1)}{(\rho_1(1)+\rho_2(1))}\Big]+\rho_2(0)\log\Big[\sfrac{2\rho_2(0)}{(\rho_1(0)+\rho_2(0))}\Big],
    \end{aligned}$$

if   $\rho_1\ll\rho_2$   and  $\infty$  otherwise.

\section{Proof of Main Results}\label{proofs}
\subsection{Change of Measure,~Exponential Tightness~and~Some General
Principles.}\label{ex-upper}\\
Denote by $\skric$ the space of bounded functions on $\skriy\times
\skriy^*$ and for $\tilde{g}\in\skric,$ we define the function
\begin{equation}\label{Def.Ug}
W_{\tilde{g}}(a)=\log \sum_{c\in \skriy^*}  e^{\tilde{g}(a,c)}\skrik\{
c\, | \, a\},
\end{equation}
for $a\in\skriy$. Using $\tilde{g}$ we define the following  new
multitype Galton-Watson tree :
\begin{itemize}
\item Assign the root $\sigma,$ type  $a\in\skriy$ according to the
probability distribution $\mu_{\tilde{g}}(a)$ given by
\begin{equation}\label{equ-U}
\mu_{\tilde{g}}(a) =\frac{\mu(a)e^{W_{\tilde{g}}(a)}}{\int\mu(db)
e^{W_{\tilde{g}}(b)} }.
\end{equation}
\item For every vertex with type $a\in \skriy$ the offspring number
and types are given independently of everything else, by the
offspring law $\tilde{\skrik}\{\,\cdot\, | \, a\}$ given by
\begin{align}\label{equ-Q}
\tilde{\skrik}\big\{ c \,\big|\, a \big\} &  =\skrik\big\{c\,\big| \,
a\big\} \exp\big(\tilde{g}\big(a,c\big)-W_{\tilde{g}}(a)\big) .
\end{align}
\end{itemize}
By $\tilde{\prob}$ we denote the transformed law and observe that
$\tilde{\prob}$ is absolutely continuous with respect to $\prob.$
Specifically,  for each finite $Y,$
\begin{align}
\frac{d\tilde{\prob}}{d\prob}(Y)& =
 \frac {e^{W_{\tilde{g}}(Y(\sigma))}}{\int e^{W_{\tilde{g}}(b)} \mu(db)}
\, \prod_{v\in V} \exp\Big[\tilde{g}(Y(v),C(v))- W_{\tilde{g}}(Y(v))\Big] \label{bddform}\\
& = \frac 1{\int e^{W_{\tilde{g}}(a)} \mu(da)}\prod_{v\in V}
\exp\Big[\tilde{g}(Y(v),C(v))-\sum_{b\in\skriy}\ell(b,\,C(v))W_{\tilde{g}}(b)
 \Big]\\
& = \frac 1{\int e^{W_{\tilde{g}}(a)}
\mu(da)}\exp\Big[\langle\tilde{g}-\sum_{b\in\skriy}\ell(b,\,\cdot)W_{\tilde{g}}(b),\,\skrim_Y\rangle
 \Big], \label{Jhatform}
\end{align}
where $$C(v)=\big(N(v),Y_1(v),\ldots,Y_N(v)\big).$$  Note, that
above  change  of  measure  appeared first in
\cite[Section~3.2]{DMS03}. Moreover, recall from \cite{DMS03} the
following results for the probability $\prob\big\{|V(T)|=n\big\}$ on
the set $S$ of integers  $n$ where the probability  is positive.

\begin{lemma}\cite[Theorem~3.1]{DMS03}\label{DMS}
Let  $T$   be a  random  tree  generated  by  a  weakly
irreducible,  critical  multitype  Galton-Watson  tree  with  finite
second  moment.  Then
$$\lim_{n\to\infty}\frac{1}{n}\log\prob\big\{|V(T)|=n\big\}=0.$$
\end{lemma}

Observe  that,  the  proof  of  Lemma~\ref{tightness} below is  a
small  adaptation  of   the  proof of  \cite[Lemma~3.2]{DMS03}.
\begin{lemma}\label{tightness}
For every $\alpha>0$ there exists a compact $\Gamma_{\alpha}
\subset\skrip(\skriy\times\skriy^*)$
with $$\limsup_{n\to\infty} \frac 1n \log \prob\big\{ \skrim_Y\not\in \Gamma_\alpha
 \, \big| \,
|V(T)|=n \big\} \le -\alpha.$$
\end{lemma}

\begin{Proof}
Let $l\in\N$,  and choose $B(l)\in\N$  large enough  such that
$\skrik\{N>B(l)|a\}\le e^{-l^2},$   for  all $a.$

Then, for  all $a,$ we have

$$\skrik\{e^{l^2\1_{\{N>B(l)\}}}|a\}=e^{l^2}\skrik\{N>B(l)|a\}+\skrik\{N\le
B(l)|a\}\le e^{l^2}\times e^{-l^2}+\skrik\{N\le B(l)|a\}\le 1+1=2.$$

 Using  the exponential  Chebyshev's  inequality we
obtain,

$$\begin{aligned}
\prob\Big\{ \skrim_Y[N>B(l)]\ge l^{-1}\, , |V(T)|=n \Big\}&\le
e^{-nl}\me\Big\{e^{l^2\sum_{v\in V}\1_{\{N(v)>B(l)\}}}\, ,
|V(T)|=n\Big\}\\
 & = e^{-ln} \me\Big\{ \prod_{v \in V(T)} \exp\big( l^2
\1_{\{N(v)>B(l)\}}\big) ,\,
|V(T)|=n \Big\} \\
& \le e^{-nl} \Big( \sup_{a\in\skriy} \skrik\big\{ \exp(l^2
1_{\{N>B(l)\}}) \, \big| \,  a\big\} \Big)^n \le e^{-n(l-\log 2)}.
\end{aligned}$$
Fix  $\alpha$   and  choose  $M>\alpha+\log 2.$  Define  the  set
$$\Sigma_M=\Big\{\pi:\, \pi[N>B(l)]<l^{-1} ,l\ge M\Big\}.$$ Observe,
$\{N\le B(l)\}\subset\skriy\times\skriy^*$ is compact,  and  so we
have  that the set $\Sigma_M$ is pre-compact in the weak topology,
by Prohorov's criterion.  Moreover,
$$\prob\big\{ \skrim_Y \not\in \Sigma_M \, \big| \, |V(T)|=n \big\} \le
\frac{1}{\prob\{|V(T)|=n\}} \frac{1}{1-e^{-1}} \exp(-n(M-\log 2)),$$
hence using Lemma~\ref{DMS} we  have
$$\limsup_{n\to\infty} \frac 1n \log \prob\big\{ \skrim_Y\not\in \Gamma_\alpha
\, \big| \, |V(T)|=n \big\} \le -\alpha,$$ for the closure
$\Gamma_\alpha$ of $\Sigma_M$. This  ends the proof  of  the  tightness
Lemma.

\end{Proof}

 We denote by $\skrip_s$ the set of all
sub-consistent measures, and by $\skrip_c$ the set of all consistent
measures in
$\Big\{(\omega,\pi)\in\tilde{\skrip}(\skriy\times\skriy)\times\skrip(\skriy\times\skriy^*): \omega_2=\pi_1\Big\}$
and notice that $\skrip_c\subseteq\skrip_s.$ For $k$ a natural
number, we denote by $\skrip_{c,k}$ the set of consistent measures in
$\Big\{(\omega,\pi)\in\tilde{\skrip}(\skriy\times\skriy)\times\skrip(\skriy\times\skriy_k^*): \omega_2=\pi_1\Big\}.$
Then, $\skrip_s$ is a closed subset of $\tilde{\skrip}(\skriy\times
\skriy)\times\skrip(\skriy\times \skriy^*)$ and $\skrip_{c,k}$ is a
closed subset of
$\Big\{(\omega,\pi)\in\tilde{\skrip}(\skriy\times\skriy)\times\skrip(\skriy\times\skriy_k^*): \omega_2=\pi_1\Big\}.$
The next two large deviation principles will help us extend LDP in
$\skrip_{c,k}$, $\skrip_s$ to
$\tilde{\skrip}(\skriy\times\skriy)\times\skrip(\skriy\times\skriy_k^*)$
and $\tilde{\skrip}(\skriy\times \skriy)\times\skrip(\skriy\times
\skriy^*)$ respectively.

\begin{lemma}\label{sub-consistent}
Suppose $Y$ is a multitype Galton-Watson tree with offspring law
$\skrik.$ Assume $(\tilde{L}_Y,\skrim_Y)$ conditioned on the event
$\{|V(T)|=n\}$ satisfies the LDP in $\skrip_s$  with convex, good
rate function
\begin{equation}\label{equ-rate3}
\widetilde{J}(\omega,\,\pi)= H(\pi \,\|\,\pi_1\otimes\skrik)
\end{equation}
Then, $(\tilde{\skril}_Y,\skrim_Y)$ conditioned on the event $\{|V(T)|=n\}$
satisfies the LDP in
$\tilde{\skrip}(\skriy\times\skriy)\times\skrip(\skriy\times\skriy^*)$
with convex, good rate function

\begin{equation}\label{equ-rate2}
J(\omega,\,\pi)=\left\{ \begin{array}{ll} H(\pi \,\|\,
\pi_1\otimes\skrik) &
\mbox{ if $(\omega,\,\pi)$ is sub-consistent and $\omega_2=\pi_1,$}\\
\infty & \mbox{ otherwise.}
\end{array} \right.\end{equation}
\end{lemma}
\begin{proof}%Suppose $Y$ is a multitype Galton-Watson tree with offspring law
%$\skrik$ and that an LDP for $(\tilde{\skril}_Y,\skrim_Y)$ conditioned on the
%event $\{|T|=n\}$ holds in  $\skrip_s,$  with convex, good rate function $\widetilde{J}.$
Observe that,
$\{|V(T)|=n\}:=\big\{\omega\in\Omega:\,|V(T)|(\omega)=n\big\}\subseteq\big\{\omega\in\Omega:\,
(\tilde{\skril}_Y,\skrim_Y)(\omega)\in\skrip_c\big\}=:\big\{(\tilde{\skril}_Y,\skrim_Y)\in\skrip_c\big\}$
and so, for all $n,$ we have $\prob\big\{ (\tilde{\skril}_Y,\,\skrim_Y) \in
\skrip_{s}\, \big| \, |V(T)|=n \big\}=1.$ Also, if
$(\omega_n,\pi_n)\in \skrip_s$ converges to $(\omega,\pi)$ then by
the Fatou's Lemma, we have that
$$\omega(a,b)=\lim_{n\to\infty}\omega_n(a,b)\ge\liminf_{n\to\infty}\sum_{c\in\skriy^*}\ell(b,c)\pi_n(a,c)\ge\sum_{c\in\skriy^*}\ell(b,c)\pi(a,c),$$
which implies $(\omega,\pi)$ is sub-consistent. This means
$\skrip_s$ is a closed subset of
$\skrip(\skriy\times\skriy)\times\skrip(\skriy\times\skriy^*).$
Therefore, by \cite[Lemma~4.1.5]{DZ98}, the LDP   for
$(\tilde{\skril}_Y,\,\skrim_Y)$ conditioned on the event $\{|V(T)|=n\}$ holds
with convex, good rate function $J.$

\end{proof}

 Recall  that  $\skrik_k$  is  offspring transition kernel
from $\skriy$ to $$\skriy_k^*=\bigcup_{n=0}^k \{n\} \times
\skriy^n$$
\begin{lemma}\label{kconsistent}
Suppose $Y$ is a multitype Galton-Watson tree with offspring law
$\skrik_k.$ Assume $(\tilde{\skril}_Y,\skrim_Y)$ conditioned on the event
$\{|V(T)|=n\}$ satisfies the LDP in $\skrip_{c,k}$  with convex,
good rate function
\begin{equation}\label{equ-rate3}
\widetilde{J}_k(\omega,\,\pi)= H(\pi \,\|\,
\pi_1\otimes\skrik_k).
\end{equation}
Then, $(\tilde{\skril}_Y,\skrim_Y)$ conditioned on the event $\{|V(T)|=n\}$
satisfies the LDP in
$\tilde{\skrip}(\skriy\times\skriy)\times\skrip(\skriy\times\skriy_k^*)$
with convex, good rate function $J_k.$
\end{lemma}
\begin{proof}Using the same argument as in the proof of
Lemma~\ref{sub-consistent} we have $\prob\big\{ (\tilde{\skril}_Y,\,\skrim_Y)
\in \skrip_{c,k}\, \big| \, |V(T)|=n \big\}=1.$ Moveover, as
$\ell(a,c)\le k$ for all $(a,c)\in\skriy\times\skriy_k^{*}$ if
$(\omega_n,\pi_n)\in \skrip_{c,k}$ converges point-wise to
$(\omega,\pi)$ then we have
$$\omega(a,b)=\lim_{n\to\infty}\omega_n(a,b)=\lim_{n\to\infty}\sum_{c\in\skriy_k^*}\ell(b,c)\pi_n(a,c)=\sum_{c\in\skriy_k^*}\ell(b,c)\pi(a,c),$$ which implies $(\omega,\pi)$ is
consistent. This means $\skrip_{c,k}$ is a closed subset of
$\Big\{(\omega,\pi)\in\tilde{\skrip}(\skriy\times\skriy)\times\skrip(\skriy\times\skriy_k^*): \omega_2=\pi_1\Big\}.$
Hence, by \cite[Lemma~4.1.5]{DZ98}, the LDP   for
$(\tilde{\skril}_Y,\,\skrim_Y)$ conditioned on the event $\{|V(T)|=n\}$ holds
with convex, good rate function $J_k$ which completes the proof of
the Lemma.
\end{proof}

In view of Lemmas~\ref{kconsistent}~and~\ref{sub-consistent}, %it is enough for us to
we establish large deviation principles in the spaces $\skrip_{c,k}$
and $\skrip_s$.
\subsection{Proof of the upper bound in Theorem~\ref{general}.}\label{s-upper}Next we derive an upper bound in a variational formulation. Denote
by $\skric$ the space of bounded functions on  $\skriy\times
\skriy^*$ and define for each $(\omega,\,\pi)$ sub-consistent
element in $\tilde{\skrip}(\skriy\times\skriy)\times\skrip(\skriy\times\skriy^*)$,
the function $\widehat{J}$ by

$$\begin{aligned}
\widehat{J}(\omega,\,\pi) :=\sup_{{g}\in \skric}\Big\{ \int
{g}(b,c)\pi(db\,, dc) &-\int U_{{g}}(b)\omega(da,\,db) \, \Big\}\label{hatJ-def}\\
&\le \sup_{{g}\in \skric}\Big\{ \int {g}(b,c)\pi(db\,, dc)
-\int\sum_{b\in\skriy} \ell(b,c)U_{{g}}(b)\pi(da\,, \,dc) \Big\}
\label{hatJ-def1}
\end{aligned}
$$

where $c=(n,a_1,\ldots,a_n)$. Note  that our  $\widehat{J}$   above
is  an  extension of  the  corresponding $\widehat{J}$  introduced
in \cite[Section~3.2]{DMS03}. Moreover,  Lemma~\ref{upper} below is
 an  extension  of  \cite[Lemma~3.5]{DMS03},  and the  strategy  of  the proof is  same  up to slight modifications.
 We recall that $\skrip_s$ is the set of all sub-consistent measures
  in
  $$\Big\{(\pi,\,\omega)\in\tilde{\skrip}(\skriy\times\skriy)\times\skrip(\skriy\times\skriy^*),\omega_1=\pi_1\Big\}.$$

\begin{lemma}\label{upper}
For each closed set $F\subset\skrip_s,$ we have
\begin{align*} \limsup_{n\to\infty} \frac 1n \log \prob\big\{
(\tilde{\skril}_Y,\,\skrim_Y) \in F\, \big| \, |V(T)|=n \big\} \le
-\inf_{(\omega,\,\pi)\in F} \widehat{J}(\omega,\,\pi).
\end{align*}
\end{lemma}
\begin{Proof}
Let $\tilde{g}\in \skric$   be  bounded  by  $M.$ Note from the
definition of $W_{\tilde{g}}$  from  \eqref{Def.Ug} that
$W_{\tilde{g}}\le M$. Using \eqref{bddform} , we  obtain

$$ e^{M}\ge \int e^{W_{\tilde{g}}(a)} \1_{\{|V(T)|=n\}}\mu(da)=\me\Big\{\exp\Big[\langle\tilde{g}-\sum_{b\in\skriy}\ell(b,\,\cdot)W_{\tilde{g}}(b),\,\skrim_Y\rangle
 \Big],\,|V(T)|=n \Big\}$$
 Now, we take limit as  $n$  approaches  infinity  and use
Lemma~\ref{DMS}   to  obtain

 \begin{equation}\label{LDP.equ1}
 \lim_{n\to\infty}\frac 1n \log\me\Big\{\exp\big[\langle\tilde{g}-\sum_{b\in\skriy}\ell(b,\,\cdot)W_{\tilde{g}}(b),\,\skrim_Y\rangle
 \big]   \Big|  |V(T)|=n \Big\}\le 0
 \end{equation}

 Similarly,  we can  use \eqref{bddform}    and  Lemma~\ref{DMS} to  obtain

\begin{equation}\label{LDP.equ2}
\lim_{n\to\infty}\frac 1n
\log\me\Big\{\exp\big[\langle\tilde{g}-W_{\tilde{g}},\,\skrim_Y\rangle
 \big]   \Big|  |V(T)|=n \Big\}\le 0.
 \end{equation}
Next,  we  write
$\widehat{J}_{\eps}(\omega,\pi):=\min\{\hat{J}(\omega,\pi),\eps^{-1}\}-\eps.$  Fix
$(\omega,\pi)\in F$ and  choice  $\tilde{g}\in\skric$, such  that
$$\big[\langle\tilde{g},\,\pi\rangle-\langle W_{\tilde{g}},\,\omega\rangle
\big]\ge \widehat{J}_{\eps}(\omega,\pi) $$

Now,  since  $\tilde{g}$  and $W_{\tilde{g}}$  are both  bounded
function, the  mapping $\langle
\tilde{g}-W_{\tilde{g}},\,\cdot\rangle$   is  continuous. We can
find open neighbourhood $B_{\omega}$, $B_{\pi}$ of $\omega$ and
$\pi$ respectively,   such that   we have
\begin{equation}\label{chebys}
\inf_{\heap{\tilde{\omega}\in B_{\omega}}{\tilde{\pi}\in
B_{\pi}}}\big[\langle\tilde{g},\,\tilde{\pi}\rangle-\langle
W_{\tilde{g}},\,\tilde{\omega}\rangle
\big]\ge\big[\langle\tilde{g},\,\pi\rangle-\langle
W_{\tilde{g}},\,\omega\rangle \big]-\eps \ge
\widehat{J}_{\eps}(\omega,\pi)-\eps
\end{equation}
Moreover,  by  the  sub-consistency   of  the  pairs
$(\tilde{\omega},\,\tilde{\pi})\in B_{\omega}\times B_{\pi}$ we
have that
\begin{equation}\label{chebys1}
\inf_{\tilde{\pi}\in
B_{\pi}}\big[\langle\tilde{g},\,\tilde{\pi}\rangle-\langle
\sum_{b\in\skriy}\ell(b,\,\cdot)W_{\tilde{g}}(b),\,\tilde{\pi}\rangle
\big]\ge\inf_{\heap{\tilde{\omega}\in B_{\omega}}{\tilde{\pi}\in
B_{\pi}}}\big[\langle\tilde{g},\,\tilde{\pi}\rangle-\langle
W_{\tilde{g}},\,\tilde{\omega}\rangle \big]\ge
\widehat{J}_{\eps}(\omega,\pi)-\eps
\end{equation}
 Applying the exponential Chebyshev inequality to \eqref{chebys1} and  using \eqref{LDP.equ1}we obtain that,
\begin{align}
\limsup_{n\to\infty} \frac 1n & \log \prob\big\{ (\tilde{\skril}_Y,
\,\skrim_Y) \in B_{\omega}\times B_{\pi} \, \big|
\, |V(T)|=n \big\} \nonumber\\
\le & \limsup_{n \to \infty}\frac 1n
\log\me\Big\{\exp\big[\langle\tilde{g}-\sum_{b\in\skriy}\ell(b,\,\cdot)W_{\tilde{g}}(b),\,\skrim_Y\rangle
 \big]   \Big|  |V(T)|=n \Big\} -  \widehat{J}_{\eps}(\omega,\pi) + \eps\nonumber\\
&\leq - \inf_{(\omega,\pi) \in F} \widehat{J}_{\eps}(\omega,\pi) +  \eps.
\label{one}\end{align}

Now we use Lemma~\ref{tightness} to choose a compact set $K_\alpha$
(for $\alpha=\eps^{-1}$) with
\begin{equation}
\limsup_{n\to\infty} \frac 1n \log \prob\big\{ \skrim_Y\not\in K_\alpha
 \, \big| \,
 |V(T)|=n \big\} \le - \eps^{-1}.
\end{equation}
For  this  $\Gamma_\alpha$  we  denote  by   $$\Sigma_{\alpha}:=\big\{
(\omega,\pi):(\omega,\pi)\in\skrip_s,\pi\in
K_\alpha\big\}.$$

The set $\Sigma_{\alpha} \cap F$ is compact and hence it may be
covered by finitely many of the sets $B_{\omega_1}\times
B_{\pi_1},\ldots,B_{\omega_m}\times B_{\pi_m}$, with $({\omega_i},
{\pi_i})\in F$ for  $i=1,\ldots,m$. Hence,
$$\begin{aligned}
\prob\big\{ (\tilde{\skril}_Y,\,\skrim_Y)\in F \, \big| \, |V(T)|=n \big\} \le
\sum_{i=1}^m \prob\big\{ (\tilde{\skril}_Y,\,\skrim_Y)& \in B_{\omega_i}\times
B_{\pi_i} \, \big| \, |V(T)|=n \big\}\\
&+ \prob\big\{
(\tilde{\skril}_Y,\,\skrim_Y) \not\in \Sigma_\alpha \, \big| \, |V(T)|=n
\big\}.
\end{aligned}$$

Using \eqref{one}  we obtain, for small enough
$\eps>0$, that
$$\begin{aligned}
\limsup_{n\to\infty} \frac 1n \log \prob\big\{& (\tilde{\skril}_Y,\,\skrim_Y)
\in F \, \big| \,|V(T)|=n  \big\}\\
&\le \max_{i=1}^m \, \limsup_{n\to\infty} \frac 1n \log \prob\big\{
(\tilde{\skril}_Y,\,\skrim_Y)
\in B_{\omega_i}\times B_{\pi_i} \, \big| \, |V(T)|=n \big\}\\
 &\le -\inf_{(\omega,\pi) \in F} \widehat{J}_\eps(\omega,\pi) + \eps.
\end{aligned}$$
 Taking $\eps \downarrow 0$ gives the required
statement.

\end{Proof}

Recall that  $\widetilde{J}:\skrip_s\to[0,\infty]$ is given by
\begin{equation}\label{equ-rate3}
\widetilde{J}(\omega,\,\pi)= H(\pi \,\|\,
\pi_1\otimes\skrik).
\end{equation}

We show that the convex rate function $\widetilde{J}$ may replace
the function $\widehat{J}$ of \eqref{hatJ-def} in the upper bound of
Lemma~\ref{upper}.
\begin{lemma}\label{rate-lsc}
The function $\widetilde{J}$ is convex and lower semicontinuous on
$\skrip_s.$ Moreover, $\widetilde{J}(\omega,\,\pi)\leq
\widehat{J}(\omega,\,\pi),$ for any $(\omega,\,\pi) \in \skrip_s.$

\end{lemma}

The  proof of the inequality $ \widetilde{J}(\omega,\,\pi)\leq
\widehat{J}(\omega,\,\pi)$  is analogous to the proof of \cite[Lemma
3.4]{DMS03}. To prove that $\widetilde{J}$ is convex, good rate
function, we consider the convex, good rate function $g:\R\to
[0,\,\infty]$ given by $g(x)=x\log x-x+1.$ Then, we can represent
the left side of \eqref{equ-rate3} in the form
\begin{align}
H(\pi\, \| \, \pi_1\otimes \skrik)=\left\{
\begin{array}{ll}\int g\circ fd(\pi_1\otimes\skrik) &
\mbox{ if $f:=\sfrac{d\pi}{d(\pi_1\otimes\skrik)}$ exists,}\\
\infty & \mbox{ otherwise.}
\end{array} \right.
\end{align}
Consequently, by \cite[Lemma~6.2.16]{DZ98}, $\widetilde{J}$ is a
convex, good rate function.

By Lemma~\ref{sub-consistent} the large deviation upper bound
Lemma~\ref{upper} holds with rate function $\widetilde{J}$  replaced
by $J.$

\subsection{Proof of Theorem \ref{general2}.}
Note  that  $\skrik_k$  is  bounded  offspring  kernel always implies
all its exponential moments are finite. But  the  converse  is  not
true. Further, the
 empirical  offspring  measure   of   the  multitype Galton-Watson  tree   with  offspring  law $\skrik_k$
 obeys  the  large  deviation  principle, \cite[Theorem~2.2]{DMS03} in  the  weak
 topology. In  Theorem~\ref{DMS1} we  give   a  modified version   of
\cite[Theorem~2.2]{DMS03}.  To do  this we recall  that  the
probability measure $\pi$  on   $\skriy\times\skriy^*$ is
shift-invariant if
$$\pi_1(a)=\sum_{(b,c)\in\skriy\times\skriy^*}\ell(a,c)\pi(b,c),\,\mbox{
for all $a\in\skriy$}.$$ %and give   a  modified version   of
%\cite[Theorem~2.2]{DMS03}  as  follows:

\begin{theorem}\cite[Theorem~2.2]{DMS03}\label{DMS1}
Suppose that $Y$ is a weakly irreducible, critical multitype
Galton-Watson tree with  offspring law  $\skrik_k$, conditioned to have
exactly $n$ vertices. Then, for $n\to\infty$, $\skrim_Y$ satisfies an LDP in $\skrip(\skriy\times\skriy^*)$ equipped
with  the  weak topology,  with speed $n$ and the convex, good rate
function
\begin{equation}\label{equ-rate1}
\Phi_k(\pi)=\left\{ \begin{array}{ll} H(\pi \,\|\, \pi_1\otimes\skrik_k) &
\mbox{ if $\pi$ is shift-invariant,}\\
\infty & \mbox{ otherwise,}
\end{array} \right.\end{equation}
$\pi_1$ is the $\skriy-$ marginal of the probability measure $\pi.$
\end{theorem}

Theorem~\ref{general2} is derived from Theorem~\ref{DMS1} by
applying the contraction principle to the linear mapping
$G:\skrip(\skriy\times\skriy_k^*)\mapsto\Big\{(\omega,\pi)\in\tilde{\skrip}(\skriy\times\skriy)\times\skrip(\skriy\times\skriy_k^*): \omega_2=\pi_1\Big\}$
given by $G(\pi)=(\omega, \pi),$ where $(\omega,\pi)$ is
consistent. Thus,  we  have   that
$$\omega(a,b)=\sum_{c\in\skriy^*}\ell(b,c)\pi(a,c),\,\mbox{
for all $a,b\in\skriy$}$$

To  be  specific, Theorem~\ref{DMS1} implies the large deviation for
$G(\skrim_Y)=(\tilde{\skril},\skrim_Y)$ with convex, good rate function
$$J_k(\omega,\pi)=\inf\Big\{\Phi_k(\pi):\pi\in\skrip(\skriy\times\skriy^*),\,G(\pi)=(\omega,
\pi),\,\mbox{ $(\omega,\pi)$ is consistent} \Big\}.$$

%\begin{equation}\label{equ-rate2}
%K_k(\pi)=\left\{ \begin{array}{ll} H(\pi \,\|\, \pi_1\otimes\skrik_k) &
%\mbox{ if \,$\displaystyle\langle
%m(\cdot,c),\,\pi(a,c)\rangle=\pi_1,$}\\\\
%\infty & \mbox{ otherwise.}
%\end{array} \right.\end{equation}
 Using shift-invariance and
consistency  of  the  pair  $(\omega,\pi)$ we have
$$\pi_1(a)=\sum_{(b,c)\in\skriy\times\skriy_k^*}\ell(a,c)\pi(b,c)=\sum_{b\in\skriy}\omega(b,a)=\omega_2(a),\,\mbox{
for all $a\in\skriy$.}$$ %which completes the proof of the theorem.
Therefore, by Lemma~\ref{kconsistent}, the LDP for $(\tilde{L},\skrim_Y)$
conditional on the event $\{|V(T)|=n\}$ holds in
$\tilde{\skrip}(\skriy\times\skriy)\times\skrip(\skriy\times\skriy_k^*)$
with convex, good rate function ${J}_k.$

\subsection{Proof of the Lower Bound in Theorem~\ref{general}}
The  global strategy of  this  proof  remains the  same  as  that of
\cite{DMS03} except  that  truncation argument  for  vertices with
too  many offsprings  and  sub-consistency  is used   in order  for
us  to  avoid the  problem of  not  having lower  semi-continuous
rate function  in  the  weak topology.  In  fact the rate
 function  of  the  lower bound  of  \ref{general} will  be  obtained  as  limit
 of  $J_k.$ i.e. the  rate function  in  Theorem~\ref{general2}. We assume  throughout  this  subsection  that  $k$  is  finite.

First, we state a lemma based  on Lemma~3.6 by Dembo etal.
\cite{DMS03}. It will help us to approximate  a measure
$\pi\in\skrip(\skriy\times\skriy_k^*)$ with $\pi_1$  strictly
positive by a shift-invariant $\pi_{x,y}.$

The detail  proof  of  our  next   Lemma  which  is  based on
the Perron-Frobenius eigen theorem and  the implicit function theorem
applied to  the function  $f(x,y)=\varrho(M_{x,y}),$ is omitted.
See, proof  of  \cite[Lemma~3.6]{DMS03}.\\

\begin{lemma}[\cite{DMS03}]\label{Approx1a} Suppose
$\pi\in\skrip(\skriy\times\skriy_k^*)$ has strictly positive
$\pi_1.$  Then, for any $y\in(0,y_0)$ and $x(y)\in(-1/2,1/2)$,  we
have
$$\lim_{\heap{x\to 0}{y\downarrow0}}\pi_{x,y}(a,c)=\pi(a,c),\,\mbox{
for all $(a,c)\in\skriy\times\skriy_k^*$}$$ and  $\pi_{x,y}$  is
shift-invariant.

\end{lemma}

\begin{proof}

To  begin, we review or
collect some notation from \cite{DMS03}. For
$\pi\in\skrip(\skriy\times\skriy_k^*)$ and $a\in\skriy$ we write
$\pi(\cdot\,|a)=\pi(a,\,\cdot)/\pi_{1}(a)$ and
$$M_{0,0}(a,b)=\sum_{c\in\skriy_k^*}\ell(a,c)\pi(c|b),\mbox{ for  $a,b
\in\skriy$}.$$   We  recall  that  $\skriy_r,$  denote the  set of  recurrent states and  $v_{0,0}$  is  the left  eigenvector  normalize  to  a  probability  vector on
$\skriy_r$ corresponding to  the  Perron-Fobenious  eigenvalue  of  $\rho(M_{0,0})= 1.$

As  $\skriy_k^*$ is  finite we can  find
$b_0\in\skriy_r,$  such  that  $\skrik_k\{c\,|\,a\}>0$ and  also
$\pi(c|b)>0$,  such  that $\sum_{a\in\skriy_r}\ell(a,c_2)$  is  large
enough to  ensure that  the difference
$\sum_{a\in\skriy_r}v_{0,0}\big[\ell(a,c_2)-\ell(a,c_1(b))\big]>0.$
  Let
$c_1(b)$  be  any number  for  $b\in\skriy_t$  and  $c_2=c_1(b)$ for
all  $b\not=b_0.$  For  any $|x|<1/2$  we   define
the probability measure $\pi_{x,0}$  as
$$\pi_{x,0}(c|b)=\pi(c|b)+x\pi(c_2|b)\pi(c_1|b)(1_{\{c=c_2\}}-1_{\{c=c_1\}}).$$
Let  $y_0=\skrik\{c_2|b_0\}\min_{b\in\skriy_r}\skrik\{c_1|b\}>0,$ and  for
any $0<y<y_0$  we define the  probability
measures $\pi_{x,y}(\cdot|b)$  by
$$\pi_{x,0}(c|b)=\min(\pi_{x,0}(c_1|b),\skrik\{c|b\}/y)\mbox{ for
$c\not=c_1$}$$
$$\pi_{x,y}(c_1|b)=\pi_{x,0}(c_1|b)+\sum_{c\not=c_1}(\pi(c|b)-\skrik\{c|b\}/y)_{+}$$
where  $_+$ indicates the  positive  part. Note that  by
construction
$$M_{x,y}(a,b)=\sum_{c\in\skriy^*}\ell(a,c)\pi_{x,y}(c|b)\to
M_{0,0}(a,b),\,\, \mbox{for any $a,b\in\skriy.$}$$

We  write
\begin{equation}\label{nu}
\pi_{x,y}(a,c)=\pi_{x,y}(c|a)(\pi_{x,y})_1(a),\,\mbox{ for
$(a,c)\in\skriy\times\skriy_k^{*},$}
\end{equation}
and  denote by $(\pi_{x,y})_1$  the  $\skriy-$marginal  of  the
shift-invariant measure $\pi_{x,y}.$

The  choice  of  $y_0$ (above) ensures   that
$\pi_{x,y}(c\,|\,b)=0$ implies $\pi(c\,|\,b)=0$ and
$\pi_{x,y}(c\,|\,b)\to \pi(c\,|\,b),$ for $y\in(0,y_0),$
$|x(y)|<\sfrac{1}{2}$   and $(a,c)\in\skriy\times\skriy_k^{*}.$ See,
\cite[Lemma~3.6~and~Proof]{DMS03}. We note  that  $(\pi_{x,y})_1$ is
the Perron-Frobenius eigen  vector   corresponding  to the
 eigen  value $\varrho(M_{x,y}).$  Hence  multiplying   through by
$(\pi_{x,y})_1(b)$ we have that $\pi_{x,y}(b,c)=0$ implies
$\pi(b,c)=0$ and $\pi_{x,y}(b,c)\to \pi(b,c),$  as required.
 Moreover, we  have  that
$$\begin{aligned}(\pi_{x,y})_1(a)=\sum_{b\in\skriy}M_{x,y}(a,b)(\pi_{x,y})_1(b)=\sum_{b\in\skriy}\sum_{c\in\skriy^*}\ell(a,c)&\pi_{x,y}(c|b)(\pi_{x,y})_1(b)
,\,\mbox{  for all $a\in\skriy$}.
\end{aligned}$$

\end{proof}

Next, we  define for every weakly irreducible, critical offspring
kernel $\skrik$ the \emph{conditional offspring law}  by
\begin{equation}\label{equ-Q}
\skrik_k\{c\,|\,a\}=\left\{
\begin{array}{ll} \sfrac{1}{\skrik\{\skriy_k^{*}\,|\,a\}}\skrik\{c\,|\,a\} & \mbox{ if $c\in\skriy_k^*$,} \\
0 & \mbox{ otherwise,}
\end{array} \right.
\end{equation}
 where
$$\displaystyle\skrik\{\skriy_k^{*}\,|\,a\}=\sum_{c\in\skriy_k^*}\skrik\{c\,|\,a\}.$$
Further,  define $\pi_k$ a probability measure on
$\skriy\times\skriy^*$ by

\begin{equation}\label{Equ-nu}
\pi_k(a,c)=\left\{
\begin{array}{ll} \sfrac{\pi(a,c)}{\|\pi\|_k} & \mbox{ if $(a,c)\in\skriy\times\skriy_k^*$,} \\
0 & \mbox{ otherwise,}
\end{array} \right.
\end{equation}
where
$\displaystyle\|\pi\|_k=\pi(\skriy\times\skriy_k^*)=\sum_{(a,c)\in\skriy\times\skriy_k^{*}}\pi(a,c)$
and we write
$$\omega_{k}(a,b):=\sum_{c\in\skriy_k^*}\ell(b,c)\pi_{k}(a,c).$$

Herein, we  note that by the dominated convergence
$$\lim_{k\to\infty}\pi(\skriy\times\skriy_k^*)=\lim_{k\to\infty}
\sum_{(a,c)\in\skriy\times\skriy^*}\1_{\{(a,c)\in\skriy\times\skriy_k^*\}}\pi(a,c)=1,$$
$$\lim_{k\to\infty}\skrik\{\skriy_k^{*}\,|\,a\}=\skrik\{\skriy^{*}\,|\,a\}=1,\,\mbox{
for all $a\in\skriy$}.$$

Denote by $(\pi_{k})_1$ the $\skriy-$marginal of the probability
measure ${\pi}_{k}$  and write for $(a,c)\in\skriy\times\skriy^*$
 $$\displaystyle
f_k(a,c):=\log\sfrac{\pi_k(a,c)} {(\pi_k)_1\otimes\skrik_k(a,\,c)},$$

$$f(a,c):=\log\sfrac{\pi(a,c)}
{\pi_1\otimes\skrik(a,\,c)},$$
$$e_k(a,c):=\sfrac{\1_{\{(a,c)\in\skriy\times\skriy_k^*\}}\skrik\{\skriy_k^{*}\,|\,a\}\pi_1(a)}
{\1_{\{c\in\skriy_k^*\}}(\pi_k)_1(a)\|\pi\|_k}.$$

Note  that  we  have  $\lim_{k\to\infty}e_k(a,c)=1,$  for  all  $(a,c)\in\skriy\times\skriy^*.$

\begin{lemma}[{\bf Limit  entropy} ]\label{Approx0}
Let  $\pi\ll\pi_1 \otimes\skrik.$  Then,  we   have

\begin{equation}\label{equ-HH}
\lim_{k\to\infty}\sum_{(a,c)\in\skriy\times\skriy^*}\pi_k(a,c)f_k(a,c)=\sum_{(a,c)\in\skriy\times\skriy^*}\pi(a,c)
f(a,c).
\end{equation}
\end{lemma}

\begin{proof}
Recall that we have assumed $\pi\ll\pi_1 \otimes\skrik$ and note that by
the definition of $\pi_k$

$$\begin{aligned}
&\sum_{(a,c)\in\skriy\times\skriy^*}\pi_k(a,c)\log\sfrac{\pi_k(a,c)}
{(\pi_k)_1\otimes\skrik_k(a,\,c)}=\sfrac{1}{\|\pi\|_k}
\sum_{(a,c)\in\skriy\times\skriy^*}\1_{\{(a,c)\in\skriy\times\skriy_k^*\}}\pi(a,c)
\log\sfrac{\pi(a,c)e_k(a,c)}
{\pi_1\otimes\skrik{(a,\,c)}}\\
&=\sfrac{1}{\|\pi\|_k}
\sum_{(a,c)\in\skriy\times\skriy^*}\1_{\{(a,c)\in\skriy\times\skriy_k^*\}}\pi(a,c)\log\sfrac{\pi(a,c)}
{\pi_1\otimes\skrik{(a,\,c)}}
+\sfrac{1}{\|\pi\|_k}\sum_{(a,c)\in\skriy\times\skriy^*}\1_{\{(a,c)\in\skriy\times\skriy_k^*\}}
\pi(a,c)\log e_k(a,c)
\end{aligned}$$

Now observe  that   $$\displaystyle
\lim_{k\to\infty}\log e_k(a,c)=\log(\lim_{k\to\infty}e_k(a,c))=\log(1)=0,\,\mbox{for
all $(a,c)\in\skriy\times\skriy^*$}.$$ Fix  $\delta>0$  and  choose
$k(\delta)\in\N,$ large enough, such that for  all  $k>k(\delta)$
we  have,
$$-\delta\le\1_{\{(a,c)\in\skriy\times\skriy_k^*\}}\log e_k(a,c)\le\delta.$$

Using  the  two  previous  inequalities  we  have,
$$\sfrac{1}{\|\pi\|_k}\sum_{(a,c)\in\skriy\times\skriy^*}\pi(a,c)f(a,c)-\sfrac{1}{\|\pi\|_k}\delta\le\sum_{(a,c)\in\skriy\times\skriy^*}\pi_k(a,c)f_k(a,c)
\le\sfrac{1}{\|\pi\|_k}\sum_{(a,c)\in\skriy\times\skriy^*}\pi(a,c)f(a,c)+\sfrac{1}{\|\pi\|_k}\delta.$$
Note  that
$\lim_{k\to\infty}\|\pi\|_k=\pi(\skriy_k^{*})=\pi(\skriy^*)=1$  and
so taking  limit  as  $k$  approaches infinity of both sides of the
above inequality, we have
$$\sum_{(a,c)\in\skriy\times\skriy^*}\pi(a,c)f(a,c)-\delta\le\lim_{k\to\infty}\sum_{(a,c)\in\skriy\times\skriy^*}\pi_k(a,c)f_k(a,c)
\le\sum_{(a,c)\in\skriy\times\skriy^*}\pi(a,c)f(a,c)+\delta.$$

Now,  allowing $\delta\downarrow 0$   we  have

$$\sum_{(a,c)\in\skriy\times\skriy^*}\pi(a,c)f(a,c)\le\lim_{k\to\infty}\sum_{(a,c)\in\skriy\times\skriy^*}\pi_k(a,c)f_k(a,c)
\le\sum_{(a,c)\in\skriy\times\skriy^*}\pi(a,c)f(a,c).$$
 which proves the
Lemma.

\end{proof}

We define the total variation metric $d$ by

\begin{equation}\label{metric}
d(\pi,\tilde{\pi})=\frac 12
\sum_{(a,c)\in\skriy\times\skriy^*}\big|\pi(a,c)-\tilde{\pi}(a,c)\big|.
\end{equation}
This metric generates the weak topology.  We recall that
$\skrip_{c,k}$ denotes  the set of consistent measures in
$\skrip(\skriy\times\skriy)\times\skrip(\skriy\times\skriy_k^*)$.
%$\skrip_{s,k}$ denotes set of sub-consistent measures in
%$\skrip(\skriy\times\skriy)\times\skrip(\skriy\times\skriy_k^*).$
In the  next  three Lemmas,  we approximate  $J(\omega,\pi)$ for a
sub-consistent  pair $(\omega,\pi)$  by $J_k(\omega_k,\pi_k) $ with
$(\omega_k,\pi_k)\in\skrip(\skriy\times\skriy)\times\skrip(\skriy\times\skriy_k^*)$
consistent. We write  $$\langle
\ell(\cdot,\cdot),\,\pi(\cdot,\cdot)\rangle(a,b):=\sum_{c\in\skriy^{*}}\ell(a,c)\pi(b,c),\mbox{
for  $a,b\in\skriy$}.$$

 For any $b\in\skriy$ we define  a  counting  measure  on $\skriy,$
 $e^{(b)}$   by $e^{(b)}(a)=0$ if $a\neq b,$ and $e^{(b)}(a)=1\mbox{ if  $a=
 b.$}$
 We write $\ell(c)=\big(\ell(a,c),a\in\skriy\big)$ and  for
  large $k,$ define  the  probability measure ${\pi}_k$  by
\begin{equation}\label{Equ.approx3}
\tilde{\pi}_k(a,c)=\pi(a,c)\Big(1-\sfrac{\|\omega\|-\|\langle
\ell(\cdot,\cdot),\,\pi(\cdot,\cdot)\rangle\|}{k}\Big)
+\sum_{b\in\skriy}\1\{\ell(c)=ke^{(b)}\}\sfrac{\omega(a,b)-\langle
\ell(\cdot,\cdot),\,\pi(\cdot,\cdot)\rangle(a,b)}{k}
\end{equation}

and  note  that  $$\displaystyle
\lim_{k\to\infty}\sum_{(a,c)\in\skriy\times\skriy_k^{*}}\tilde{\pi}_k(a,c)=1.$$
 For  large $k$ with
$\displaystyle
\|\tilde{\pi}_k\|_k:=\sum_{(a,c)\in\skriy\times\skriy_k^{*}}\tilde{\pi}_k(a,c)>0,$
define  another  probability  measure $\hat{\pi}_k$  by
\begin{equation}\label{Equ.approx3}
\hat{\pi}_k(a,c)=\sfrac{1}{\|\tilde{\pi}_k\|_k}\tilde{\pi}_k(a,c),\,\mbox{if
$(a,c)\in\skriy\times\skriy_k^{*}$ and  $\hat{\pi}_k(a,c)=0$
otherwise.}
\end{equation}

From $\hat{\pi}_k$  we define a finite  measure
$\hat{\omega}_k\in\skrip(\skriy\times\skriy)$ by
 \begin{equation}\label{Equ.approx4}
 \hat{\omega}_k(a,b)=\sum_{c\in\skriy^*}\ell(a,c)\hat{\pi}_k(b,c).\end{equation}

 We prove  that  $(\hat{\omega}_k, \hat{\pi}_k)$  is  a  consistent
 approximation of   $(\omega,\pi)$ and $\pi_k$  is  an  element  of
 $\skrip(\skriy\times\skriy_k^*).$\\

 We shall  henceforth assume  that  $(\pi_k)_1(a)>0$   for  all  $a\in\skriy$   otherwise  if  $(\pi_k)_1(a)=0$  for  some $a\in\skriy$ then  as  $\skrik_k$  is  weakly  irreducible,  critical  offspring  kernel  we  can  find  strictly  positive  probability  vector  $(\pi_k)_0$  such that  $$\pi_{k}^{*}(a,c)=\skrik_k\{c\,|\,a\}(\pi_k)_0(a)$$    If  we  fix  $0<\eps<1,$ we can   define  another   probability  measure  $\pi_{k}^{\eps}(a,c)=(1-\eps)\pi_k(a,c)+\eps\pi_{k}^{*}(a,c)$  such  that  $(\pi_{k}^{\eps})_1$  is  strictly  positive. Refer  to  \cite[Lemma~3.6 and Proof]{DMS03}  for  similar argument  for  shift-invariant  measures.

\begin{lemma}[{\bf Consistent Approximation} ]\label{Approx111}
Let $(\omega,  \pi)\in\skrip_s.$ %be  such that $\omega_2=\pi_1.$
Then,  we  have
\begin{itemize}
\item [(i)]  $(\hat{\omega}_k,  \hat{\pi}_ k)$    is    consistent % and  belongs to the  set $\skrip_{s,k}\subset\skriy_s$$.
\item [(ii)]  $(\hat{\omega}_k,  \hat{\pi}_ k)  \to (\omega,  \pi)$  as
$k\to\infty.$
\item [(iii)] $(\hat{\omega}_k)_2=(\hat{\pi}_ k)_1.$
\end{itemize}
\end{lemma}
\begin{proof}
%To  begin,  note  that  the  contruction of  the probability  measure  $\hat{\pi}_k$ required  that  the  pair  %$(\omega,  \pi)\in\skrip_s$ because the measure $\tilde{\pi}_k,$  see \eqref{Equ.approx3},  used  to  construct %$\hat{\pi}_k$ is  a  probability  measure  only  when  $(\omega,  \pi).$

 (i) (ii) Fix $\eps>0$, write $\delta=\sfrac{2\eps}{3}$  and  choose $k(\delta)\in\N$
 (large)
such that, for  all  $k>k(\delta),$ $\displaystyle
\|\tilde{\pi}_k\|_k>0,$

\begin{equation}\sfrac{\|\omega\|-\|\langle
\ell(\cdot,\cdot),\,\pi(\cdot,\cdot)\rangle\|}{k\|\tilde{\pi}_k\|_k}\le
\delta \,\,\mbox{and}\,\, \big|\sfrac{1}{\|\tilde{\pi}_k\|_k}
-1\big|\le \delta.
\end{equation}

 Using  the  triangle  inequality, we  have  that

 $$\begin{aligned} d(\hat{\pi}_{k},\pi)&=\frac12
\sum_{(a,c)\in\skriy\times\skriy^*}\big|\hat{\pi}_{k}(a,c)-\pi(a,c)\big|\\
&\le\frac12\big\|\sfrac{1}{\|\tilde{\pi}_k\|_k}
-1\big\|+\sfrac{\|\omega\|-\|\langle
\ell(\cdot,\cdot),\,\pi(\cdot,\cdot)\rangle\|}{2k\|\tilde{\pi}_k\|_k}
+\sfrac{1}{2\|\tilde{\pi}_k|_k}\sum_{b\in\skriy}\sum_{(a,c)\in\skriy\times\skriy^*}\1\{\ell(c)=ke^{(b)}\}\sfrac{\omega(a,b)-\langle
\ell(\cdot,\cdot),\,\pi(\cdot,\cdot)\rangle(a,b)}{k}\\
 &=\frac{1}{2}\big\|\sfrac{1}{\|\tilde{\pi}_k\|_k}
-1\big\|+\sfrac{\|\omega\|-\|\langle
\ell(\cdot,\cdot),\,\pi(\cdot,\cdot)\rangle\|}{2k\|\tilde{\pi}\|_k}+\sfrac{\|\omega\|-\|\langle
\ell(\cdot,\cdot),\,\pi(\cdot,\cdot)\rangle\|}{2k\|\tilde{\pi}\|_k}\\
&\le \sfrac{\delta}{2}+\delta\\
&=\eps. \end{aligned}
$$
  Moreover, for  all $a,b\in\skriy,$ we have

$$
\begin{aligned}
&\hat{\omega}_k(a,b)\\
&=\sum_{c\in\skriy^*}\ell(a,c)\hat{\pi}_k(b,c)\\
&=\sfrac{1}{{\|\tilde{\pi}_k\|_k}}\Big(1-\sfrac{\|\omega\|-\|\langle
\ell(\cdot,\cdot),\,\pi(\cdot,\cdot)\rangle\|}{k}\Big)
\sum_{c\in\skriy^*}\ell(a,c)\pi(b,c)+\sfrac{1}{{\|\tilde{\pi}_k\|_k}}\omega(a,b)-\sfrac{1}{{\|\tilde{\pi}_k\|_k}}\langle \ell(\cdot,\cdot),\,\pi(\cdot,\cdot)\rangle(a,b)\\
&=\sfrac{1}{{\|{\tilde{\pi}_k\|_k}}}\omega(a,b)-\sfrac{1}{{\|{\tilde{\pi}}_k\|_k}}\sfrac{\|\omega\|-\|\langle
\ell(\cdot,\cdot),\,\pi(\cdot,\cdot)\rangle\|}{k}\langle
\ell(\cdot,\cdot),\,\pi(\cdot,\cdot)\rangle(a,b)
\stackrel{k\uparrow\infty}{\longrightarrow} \omega(a,b).
\end{aligned}$$

This proves that   $(\hat{\omega}_k,\hat{\pi}_k)$   is  a consistent
 element  of  $\skrip(\skriy\times\skriy)\times\skrip(\skriy\times\skriy_k^*)$   converging to
 $(\omega,\pi),$ sub-consistent.

(iii) We use   Lemma~\ref{Approx1a}  to choose shift-invariant
$\hat{\pi}_{k,x,y}$  converging  to  $\hat{\pi}_k.$
 Using shift-invariance of  $\hat{\pi}_{k,x,y}$ we  have
 $$\begin{aligned}
 (\hat{\omega}_k)_2(b)=\sum_{a\in\skriy}\hat{\omega}_k(a,b)=\sum_{a\in\skriy}\sum_{c\in\skriy_k^*}\ell(b,c)\lim_{\heap{x\to
0}{y\downarrow0}}\hat{\pi}_{k,x,y}(a,c) & =\lim_{\heap{x\to
0}{y\downarrow0}}\sum_{a\in\skriy}\sum_{c\in\skriy_k^*}\ell(b,c)\hat{\pi}_{k,x,y}(a,c)\\
&=\lim_{\heap{x\to
0}{y\downarrow0}}(\hat{\pi}_{k,x,y})_1(b)=(\hat{\pi}_k)_1(b),
\end{aligned}$$
where  $(\hat{\pi}_{k,x,y})_1$  is  the  $\skriy-$ marginal  of
$\hat{\pi}_{k,x,y}.$
 This ends  the  proof  of  the  Lemma.
\end{proof}

 For $\skrik_k$ we
recall the definition of the rate function
${J}_k:\skrip(\skriy\times\skriy)\times\skrip(\skriy\times\skriy_k^*)\to[0,\infty]$
from Theorem~\ref{general2} as
$$J_k(\omega,\,\pi)=\left\{ \begin{array}{ll} H(\pi \,\|\,
\pi_1\otimes\skrik_k) &
\mbox{ if $(\omega,\,\pi)$ is consistent and $\omega_2=\pi_1,$}\\
\infty & \mbox{ otherwise.}
\end{array} \right.$$
Lemma~\ref{Approx2} below  is a key ingredient in  our proof  of the
lower bound in Theorem~\ref{general} and will be proved using the
above two approximation Lemmas.

\begin{lemma}[{\bf Rate Function Approximation }]\label{Approx2}
 Suppose $(\omega,  \pi)\in\skrip_s$  and
 $\pi \ll \pi_1 \otimes \skrik.$
 Then, for every $\eps>0,$ there exists $(\hat{\omega},
 \hat{\pi})\in\skrip_{c,k}$ such that
 $|\omega(a,b)-\hat{\omega}(a,b)|<\eps,$ for all $a,b\in\skriy$,
 $d( \pi,\hat{\pi}) \le  \eps,$ $\hat{\omega}_2=\hat{\pi}_1$ and $$\widetilde{J}_k(\hat{\omega},
 \hat{\pi})- \widetilde{J}(\omega, \pi)\le\eps.$$

\end{lemma}

\begin{proof}
Recall from \eqref{Equ.approx3}  and \eqref{Equ.approx4}  the
definitions of $\hat{\pi}_k,$ $\hat{\omega}_k.$  Note  from
Lemma~\ref{Approx111} (i) and (ii)  that
 $(\hat{\omega}_k,\hat{\pi}_k)$  is consistent pair  of  measures converging to
 $(\omega,\pi),$  that  satisfies  all assumptions  of
 Lemma~\ref{Approx2}. Furthermore , we  have
 $(\hat{\omega}_k)_2=(\hat{\pi}_k)_1$  by Lemma~\ref{Approx111}
 (iii).

Now, we take $\hat{\pi}_k,$ $\hat{\omega}_k$ in Lemma~\ref{Approx2}
as
 $\hat{\pi}=\hat{\pi}_k$ and $\hat{\omega}=\hat{\omega}_k.$  Then,
$\hat{\pi} \ll (\hat{\pi})_1 \otimes \skrik_k,$
 and by Lemma~\ref{Approx111},  we  have   for every $\eps>0,$ $|\omega(a,b)-\hat{\omega}(a,b)|<\eps,$ for all
 $a,b\in\skriy$, and
 $$d( \pi,\hat{\pi}) \le  \eps.$$   Now, using Lemma~\ref{Approx0} we obtain
$$\lim_{k\to\infty}\widetilde{J}_k(\omega_{k},\pi_{k})=\lim_{k\to\infty}\sum_{(a,c)\in\skriy\times\skriy^*}
\pi_{k}(a,c)\log\sfrac{\pi_{k}(a,c)} {(\pi_{k})_1\otimes\skrik_k(a,\,c)}
= \sum_{(a,c)\in\skriy\times\skriy^*}\pi(a,c) \log\sfrac{\pi(a,c)}
{\pi_1\otimes\skrik(a,\,c)}=\widetilde{J}(\omega,\pi).$$

\end{proof}

 We  recall that
$C(v)=(N(v),\,Y_1(v),\,\ldots,\,Y_{N(v)})$ and note that, for every
$k$ such that $\displaystyle\min_{a\in\skriy}\skrik\{\skriy_k^*|a\}>0$
and   any tree-indexed process $Y$, we have that

\begin{equation}\label{Equ-lower}
\begin{aligned}
\prob\big\{ Y=x \, \big| \,
 |T|=n \big\}&\ge\prob\big\{
(Y=x,\,C(v)\in\skriy_k^*,\,v\in V\,\big| \, |V(T)|=n\big\}\\
&=\prod_{v\in
V(T),\,|V(T)|=n}\skrik\{\skriy_k^*|x(v)\}\times\prob_k\big\{Y=x
\big| \, |V(T)|=n\big\}\\
&\ge\big(\min_{a\in\skriy}\skrik\{\skriy_k^*|a\}\big)^n\times\prob_k\big\{Y=x\,
\big| \, |V(T)|=n\big\},
\end{aligned}
\end{equation}
where $\prob_k$ denote the law of the tree-indexed process %$\tilde{X}$
with initial distribution $\rho$ and offspring kernel $\skrik_k,$ and
$$\lim_{k\to\infty}\min_{a\in\skriy}\skrik\{\skriy_k^*|a\}
=\lim_{k\to\infty}\min_{a\in\skriy}\sum_{c\in\skriy_k}\skrik\{c|a\}
=\lim_{k\to\infty}\sum_{c\in\skriy^*}\1_{\{c\in\skriy_k\}}\min_{a\in\skriy}\skrik\{c|a\}=1,$$
since $\skriy$ is a finite Alphabet. To complete the proof of the
lower bound , we take $O\subset\skrip_s$. Then, for any
$(\omega,\pi)\in O$ sub-consistent with $\omega_2=\pi_1,$
$\pi\ll\pi_1 \otimes\skrik$ we may find $\eps>0$ with ball around
$(\omega,\pi)$  of radius $2\eps$ contained in $O.$ By our
approximation Lemma~\ref{Approx2}, we may find
$(\omega_{k},\pi_{k})\in O\bigcap\skrip_{c,k}$ with
$|\omega_{k}(a,b)-\omega(a,b)|\downarrow 0, $
$d(\pi_{k},\pi)\downarrow 0 ,$  $(\omega_{k})_2=(\pi_{k})_1,$
$\pi_k\ll(\pi_k)_1 \otimes\skrik_k$ and $$\tilde{J}_k(\omega_{k},\pi_{k})-
\tilde{J}(\omega,\pi)\le \eps.$$ Hence, using the lower bound of
Theorem~\ref{general2} for offspring kernel $\skrik_k$ given by
\eqref{equ-Q},  \eqref{Equ-lower} for large $k\geq k(\eps)$ (with
$\displaystyle\min_{a\in\skriy}\skrik\{\skriy_k^*|a\}>0$) and for large
$n\geq n(\eps),$ we obtain
$$
\begin{aligned}
\prob\big\{ (\tilde{\skril}_Y,\,\skrim_Y) \in O\,\big| \, |V(T)|=n \big\}
&\ge\prob\big\{\,|\omega_{k}(a,b)-\tilde{\skril}_Y(a,b)|<\eps,\,\mbox{
$\forall a,b\in\skriy$,}\, d(\pi_{k},\skrim_Y)<\eps \,\big| \, |V(T)|=n \big\}\\
&\ge
e^{n\alpha_k}\times\prob_k\big\{\,|\omega_{k}(a,b)-\tilde{\skril}_Y(a,b)|<\eps,\,\mbox{
$\forall a,b\in\skriy$,}\,\\
 &\quad\quad \quad\quad \quad \quad \quad \quad \quad\quad d(\pi_{k},\skrim_Y)<\eps \,\big| \, |V(T)|=n \big\}\\
&\ge \exp\big(-n(\widetilde{J}_k(\omega_{k},\pi_{k})+\eps-\alpha_k)\big)
\end{aligned}
$$

where $\alpha_k=\log(\min_{a\in\skriy}\skrik\{\skriy_k |a\}).$ Taking
limits we  have that
$$\displaystyle\liminf_{n\to\infty} \frac 1n \log
\prob\big\{ (\tilde{\skril}_Y,\,\skrim_Y) \in O\,\big| \, |V(T)|=n
\big\}\ge-\lim_{k\to \infty}\widetilde{J}_k(\omega_{k},\pi_{k})-\lim_{k\to
\infty}\alpha_k-\eps\ge-\widetilde{J}(\omega,\pi)-2\eps
$$ Taking
$\eps\downarrow 0$ we have the desired result which completes the
proof of the lower bound.

Applying  Lemma~\ref{sub-consistent} to  the  LDP for $(\tilde{L},\skrim)$ in the space $\skrip_s$  with  rate  function  function $\widetilde{J}$  we  obtain  the LDP  for  $(\tilde{L},\skrim)$  in the  whole  space $\tilde{\skrip}(\skriy\times\skriy)\times\skrip(\skriy\times\skriy^*)$  with  rate function $J.$

\section{Proof of
Corollaries~\ref{general-cor},\,\ref{geometric}~and~Theorem~\ref{main}}\label{cormain}

\subsection{Proof of Corollary~\ref{general-cor}.}
We derive this corollary from Theorem~\ref{general} by applying the
contraction principle to the linear mapping
$W:\tilde{\skrip}(\skriy\times\skriy)\times\skrip(\skriy\times\skriy^*)\mapsto\skrip(\skriy\times\skriy^{*})$
defined by $$W(\omega,\pi)(a,c)=\pi(a,c),\,\mbox{ for all
$(a,c)\in\skriy\times\skriy^*$}.$$ In fact  Theorem~\ref{general}
implies the large deviation principle for $W(\tilde{\skril}_Y,\skrim_Y)$ with
convex, good rate function
$\widehat{K}(\pi)=\inf\big\{J(\omega,\pi):\,W(\omega,\pi)=\pi\big\}.$
Now, using sub-consistency and $\omega_2=\pi_1$  we obtain the form
$\widehat{K}(\pi)=H(\pi \,\|\, \pi_1\otimes\skrik),$ for $\pi$  weak
shift-invariant. We write
$$\skrip_2=\Big\{\pi:\,\pi\in\skrip(\skriy\times\skriy^*),\,\pi \, \mbox{is  weak-  shift-invariant} \Big\}.$$
 Also, for all  (values of ) $n$ where $\prob\{|V(T)|=n\}>0,$ we
have
$$\prob\big\{\skrim_Y\in\skrip_2\,||V(T)|=n\big\}=1.$$

Moreover, if $\pi_n\in\skrip_2$ converges to $\pi$ then
 $$\pi_1(a)=\lim_{n\to\infty}(\pi_n)_1(a)\ge\liminf_{n\to\infty}
 \sum_{(b,c)\in\skriy\times\skriy^*}\ell(a,c)\pi_n(b,c)\ge\sum_{(b,c)\in\skriy\times\skriy^*}\ell(a,c)\pi(b,c),$$
which implies $\pi$ is weak-shift  invariant. This means $\skrip_2$
is a closed subset of $\skrip(\skriy\times\skriy^*)$. Therefore, by
\cite[Lemma~4.1.5]{DZ98}, the LDP for $\skrim_Y$ conditional on the event
$\{|V(T)|=n\}$ holds with convex, good rate function $K,$ which
completes the proof of the corollary.
\subsection{Proof of Theorem~\ref{main}.}\label{s-main} We begin the
proof of the theorem by stating the following Lemma,  which appears
in \cite[(3.40)]{DMS03}. The main  difference between
\cite[(3.40)]{DMS03} and  \eqref{Equ-inf}  lies  in  the  proof  of
the fact that  $\Lambda^*(z)=\phi_p(z)$,  were slight adaptation  is
made.  We  write  $|g|=\sum_{b\in\skriy}g(b).$
\begin{lemma} Suppose that  $q(n,a_1,\ldots,a_n)=p(n) \prod_{i=1}^n \widehat{q} (a_i)$, where $\widehat{q}(\cdot)$ is a probability
vector on $\skriy$ and $p(\,\cdot\,)$ a probability measure with
mean one on the nonnegative integers. Then, we have
\begin{equation}\begin{aligned}\label{Equ-inf}
\inf\Big\{ H(\widetilde{\pi} \,\|\,  q)\, : \, \widetilde{\pi} \in
\skrip(\skriy^*), \;\; g(b)&=\sum_{c\in\skriy^*} \ell(b,c) \,
\widetilde{\pi} (c) \mbox{ for all } b \in \skriy \Big\}\\
& = |g|
H\Big(g|g|^{-1}\,\|\,  \widehat{q}\Big) + \phi_p\big(|g|\big),
\end{aligned}
\end{equation}

where $g:\skriy\to\R $  and  $c=(n,a_1,\ldots,a_n).$
\end{lemma}
\begin{proof}

For $ \widetilde{\pi} \in \skrip(\skriy^*),$ we let
$g(b)=\sum_{c\in\skriy^*} \ell(b,c) \, \widetilde{\pi}(c), \mbox{
for all } b \in \skriy $ and  suppose first that $|g|=0$, i.e.
$g(b)=0$ for all $b \in \skriy$. Then,
$\widetilde{\pi}((0,\emptyset))=1$ is the only possible measure in
left side of  \eqref{Equ-inf},  which gives  us  $-\log q((0,\emptyset))= -\log p(0)$. It follows from \eqref{Ip-def}
that $\phi_p(0)=-\log p(0)$ giving  us \eqref{Equ-inf} for such
$g(\cdot)$. We assume hereafter that $z>0$. Now the possible
measures $\widetilde{\pi}(\cdot)$ in the left side of
\eqref{Equ-inf} are of the form $\widetilde{\pi}(c)=\alpha(n)
u_n(a_1,\ldots,a_n)$ for $c=(n,a_1,\ldots,a_n)$, with $u_0=1$, where
$u(\cdot)$ is a probability measure on the nonnegative integers
whose mean is $z$, and $u_n(\,\cdot\,)$, $n \geq 1$, are probability
measures on $\skriy^n$ with marginals $u_{n,i}(\,\cdot\,)$ such that
\begin{equation}\label{Ia-cons}
g(b)=\sum_{n=1}^\infty \alpha(n) \sum_{i=1}^n u_{n,i} (b) \quad \mbox{
for all } b \in\skriy \;.
\end{equation}
By the assumed structure of $q(\,\cdot\,)$ we have for such
$\widetilde{\pi}(\,\cdot\,)$ that
\begin{equation}
H(\widetilde{\pi} \,\|\,  q)  = \sum_{n=1}^\infty \alpha(n) H ( u_n \, \|
\, \widehat{q}^n ) + H( u \, \| \, p)\; ,
\end{equation}
where $\widehat{q}^n$ denotes the product measure on $\skriy^n$ with
equal marginals $\widehat{q}$. Write $|g|:=\sum_{b\in\skrix}g(b)$ and   recall   from \cite{DMS03}
Subsection~3.4 that
$$ |g|H\,\Big( |g|^{-1}
\sum_{n=1}^\infty \alpha(n) \sum_{i=1}^n u_{n,i}\,\Big\|\,\widehat{q}\Big)\leq
\sum_{n=1}^\infty \alpha(n) \sum_{i=1}^n H\big(u_{n,i}\,\big\|\,
\widehat{q}\big)
 \leq \sum_{n=1}^\infty \alpha(n) H \big( u_n \, \big\| \, \widehat{q}^n \big)\,,$$ with equality whenever $u_n = \prod_{i=1}^n
u_{n,i}$ and $u_{n,i}$ are independent of $n$ and $i$.  So, in view
of \eqref{Ia-cons},
\begin{equation}\label{Ia-temp}
 |g| H(g|g|^{-1} \,\|\,  \widehat{q}) +
H( s \, \| \, p) \le H(\widetilde{\pi} \,\|\,  q) \,,
\end{equation}
with equality when  $u_n = (|g|^{-1} g)^n$ for all $n \geq 1$.

Now,write $\displaystyle\Lambda_p(\lambda):= \log\sum_{n}e^{\lambda
n}p(n)$ and notice that $\Lambda$ convex  function and
$\Lambda(0)=0<\infty,$ and so, we have, for every $\lambda\in\R,$\,
$\Lambda_p(\lambda)>-\infty.$ Using Jensen's inequality, for every
$s\in\skrip(\N\cup\{0\})$  and every $\lambda\in\R,$ we have
$$
\Lambda_p(\lambda)= \log\sum_{n}\alpha(n)\big(\sfrac{e^{\lambda
n}p(n)}{\alpha(n)}\big)\ge\sum_{n}\alpha(n)\log\big(\sfrac{e^{\lambda
n}p(n)}{\alpha(n)}\big)=\lambda\sum_{n}n\alpha(n)-H( s \, \| \, p),
$$
with equality if $s_{\lambda}(n)=p(n)e^{\lambda
n-\Lambda(\lambda)}.$ Thus, for all $\lambda$ and all $z,$ we have
\begin{equation}\label{Equ-Lamb}
\lambda z-\Lambda(\lambda)\le  \inf \big\{ H( s \, \| \, p) :
s\in\skrip(\N\cup\{0\})\mbox{ and $\sum_{n}\alpha(n) n = z $}
\big\}:=\Lambda^*(z),
\end{equation}
 with equality when $\sum_{n} \alpha(n) n = z .$ Elementary calculus also shows that
\begin{equation}\label{Equ-Ip}
\Lambda^*(z)=\lambda_* z-\Lambda_p(\lambda_*),
\end{equation}
where $\lambda_*$ is the solution of $\Lambda_p^{'}(\lambda_*)=z$
and $\sfrac{d\Lambda_p}{d\lambda}:=\Lambda_p^{'}(\lambda).$
Combining \eqref{Equ-Ip} and \eqref{Equ-Lamb} we obtain
$$\sup_{\lambda\in\R}\big\{\lambda z-\Lambda_p(\lambda)\big\}\le\Lambda^*(z)
\le\sup_{\lambda\in\R,\,\Lambda_p^{'}(\lambda)=z}\big\{\lambda
z-\Lambda(\lambda)\big\}\le\sup_{\lambda\in\R}\big\{\lambda
z-\Lambda_p(\lambda)\big\}.$$ This yields $\Lambda^*(z)=\phi_p(z),$
which ends the proof of the Lemma.
\end{proof}

Next, note that $Y$ is an irreducible, critical multitype
Galton-Watson tree with offspring law
\begin{equation}
\skrik\{c\,|\,b\}=p(n)\prod_{i=1}^{n} K\{a_i\,|\,b\}, \mbox{  for
$c=(n,a_1,\ldots,a_n)$.}
\end{equation}

We derive Theorem~\ref{main} from
Theorems~\ref{general}~and~\ref{general2} by applying the
contraction principle to the continuous linear mapping
$F:\skrip(\skriy \times \skriy)\times\skrip(\skriy \times \skriy^*)
\to \R^{\skriy \times \skriy}$, defined by
\begin{equation}
F(\omega,\pi)(a,b)=\omega(a,b),\, \mbox{ for all } (\omega,\pi) \in
\skrip(\skriy \times \skriy)\times\skrip(\skriy \times \skriy^*)
\mbox{ and } a,b\in\skriy.
\end{equation}
It is easy to see that on $\{ |V(T)| = n \}$ we have
$\skril_Y=\frac{n}{n-1} F(\tilde{\skril}_Y,\skrim_Y)=\frac{n}{n-1}\tilde{\skril}_Y$. It
follows that conditioned on $\{ |V(T)| = n \}$ the random variables
$\skril_Y$ are exponentially equivalent to $\tilde{\skril}_Y$, hence $L_X$
satisfy the same large deviation principle as $F(\tilde{\skril}_Y,\skrim_Y),$
see \cite[Theorem 4.2.13]{DZ98}.Without loss of generality we
restrict the space for the large deviation principle of $\skril_Y$ to the
set of all probability vectors on $\skriy \times \skriy$, see
\cite[Lemma 4.1.5(b)]{DZ98}.

Suppose $p$ has finite second moment. Then, Theorem \ref{general}
implies the large deviation principle for $F(\tilde{\skril}_Y,\skrim_Y)$
conditioned on $\{ |V(T)|=n \}$ with the good rate function
$\phi(\rho)=\inf\{J(\rho,\,\pi) : F(\rho,\pi)= \rho\}$, see for example
\cite[Theorem 4.2.1]{DZ98}. Convexity of $I$ follows easily from the
linearity of $F$ and convexity of $J$.

Turning to the proof of \eqref{Idef}, recall that $\pi$ is
sub-consistent if and only if $
F(\omega,\pi)(a,b)\ge\sum_{c\in\skriy^*}\ell(b,c)\pi(a,c)$ for all $a,b
\in \skriy$. Hence, we have that
\begin{equation}
\phi(\rho,\pi) = \inf\big\{ H(\pi \, \| \pi_1\otimes\skrik)\, : \,
F(\rho,\pi)(\cdot,\cdot)\ge\sum_{c\in\skriy^*}\ell(\cdot,\,c)\pi(\cdot,c),
\pi_1=\rho_2 \big\} \;.
\end{equation}
Note that $\pi_1(a)=0$ yields $\sum_b F(\pi)(a,b)=0$ \,if\,
$\displaystyle
F(\rho,\pi)(\cdot,\cdot)=\sum_{c\in\skriy^*}\ell(\cdot,\,c)\pi(\cdot,c)$
and
$$\sum_{(b,c)\in\skriy^*}\ell(b,c)\pi(a,c)<0\,\,\mbox{if \,$ \displaystyle
F(\rho,\pi)(\cdot,\cdot)>\sum_{c\in\skriy^*}\ell(\cdot,c)\pi(\cdot,c).$}$$
Hence if $\rho_1(a)>0=\rho_2(a)$ for some $a \in \skriy$  then
$$\displaystyle\Big\{\pi :
F(\rho,\pi)(\cdot,\cdot)=\sum_{c\in\skriy^*}\ell(\cdot,c)\pi(\cdot,c),
\pi_1 = \rho_2\Big\}\cup\Big\{\pi :
F(\omega,\pi)(\cdot,\cdot)>\sum_{c\in\skriy^*}\ell(\cdot,c)\pi(\cdot,c),
\pi_1 = \rho_2\Big\}$$ is an empty set, and therefore
$\phi(\rho)=\infty$. Assuming, throughout the rest of the proof that
$\rho_1 \ll \rho_2$, it is not uneasy to verify that
\begin{equation}\label{I-abs}
\phi(\rho) = \sum_{a\in\skriy} \rho_2(a) \, \widehat{\phi}
\Big(\frac{\rho(a,\cdot)}{\rho_2(a)},\skrik\{\,\cdot\,|\,a\}\Big) \;,
\end{equation}
where for $q \in \skrip(\skriy^*)$,
$\displaystyle\widehat{\phi}\Big(\frac{\rho(a,\cdot)}{\rho_2(a)},q\Big)=
\inf\Big\{\widetilde{\phi}(g,q):\,g:\skriy \to \reals_+,\,
g(a)\le\sfrac{\rho(a,\cdot)}{\rho_2(a)},\,\mbox{ for all
$a\in\skriy$}\Big\}$  and %We write for $q \in \skrip(\skriy^*),$
\begin{equation}\label{Ia-dec}
\widetilde{\phi}(g,q) := \inf\Big\{ H(\widetilde{\pi} \,\|\,  q)\, :
\, \widetilde{\pi} \in \skrip(\skriy^*), \;\;
g(b)=\sum_{c\in\skriy^*} \ell(b,c) \, \widetilde{\pi} (c) \mbox{ for
all } b \in \skriy \Big\} \;.
\end{equation}

Suppose now that $q(c)=p(n) \prod_{i=1}^n \widehat{q} (a_i)$ for all
$c=(n,a_1,\ldots,a_n)$, where $\widehat{q}(\cdot)$ is a probability
vector on $\skriy$ and $p(\,\cdot\,)$ a probability measure with
mean one on the nonnegative integers, whose  second moment is
finite. Then, by Lemma~\ref{Equ-inf}, we have the representation
\begin{equation}\label{Ia-iden}
\widetilde{\phi}(g,q) = |g| H\big(g|g|^{-1} \,\|\,
\widehat{q}\big) + \phi_p(g) \;,
\end{equation}
where $|g|:=\sum_{b\in\skriy}g(b).$  Therefore, it suffice
for us to show that
\begin{equation}\label{Ia-iden1}
\inf\Big\{\widetilde{\phi}(g,\widehat{q}):\,\, g:\skriy \to
\reals_+,\, g(b)\le \sfrac{\rho(a,b)}{\rho_2(a)},\,\mbox{for all
$a\in\skriy$}\Big\}=\sfrac{\rho_1(a)}{\rho_2(a)}
H\big(\sfrac{\rho(a,\cdot)}{\rho_1(a)}\,\|\, \widehat{q}\big) +
\phi_p(\sfrac{\rho_1(a)}{\rho_2(a)}).
\end{equation}

To do this, we write $$h(g(b)):=\widetilde{\phi}(g,\widehat{q})+
\alpha(b)\big(g(b)-\sfrac{\rho(a,b)}{\rho_2(a)}\big),\,\mbox{ for
$b\in\skriy,$}$$  where $\alpha$ is a Lagrange multiplier. Then,
elementary calculus shows that $\alpha(b)$ is the solution of the
equation
$$\phi_p(\sfrac{\rho_1(a)}{\rho_2(a)})-\sfrac{\rho(a,b)}{\rho_2(a)}\sum_{a\in\skriy}e^{-\alpha(a)}\widehat{q}(a)=0$$
and that $g(b)=\sfrac{\rho(a,b)}{\rho_2(a)}$ is the minimizer of
our constraint optimization problem.  Writing
$g(b)=\sfrac{\rho(a,b)}{\rho_2(a)}$ in \ref{Ia-iden} we obtain left
side of \eqref{Ia-iden1} which proves the theorem in case of $p$
with unbounded support and finite second moments.

\subsection{Proof of Corollary~\ref{geometric}.}
Recall that $T$ is Galton-Watson tree with offspring law
$p(n)=2^{-(n+1)},$ $n=0,\ldots,.$ Also, we recall that $Y$
is markov chain indexed by $T$ with arbitrary initial distribution
and transition kernel $K.$ Then, $Y$ satisfies all assumptions of
Theorem~\ref{main}, in particular we have
$\sum_{n=0}^{\infty}n^2p(n)=3<\infty.$ Therefore, by
Theorem~\ref{main}, $\skril_Y$ conditioned on the events $\{|V(T)|=n\}$
satisfies a large deviation principle in
$\skrip(\skriy\times\skriy)$ with good, convex rate function

\begin{align}\label{Equ-geo}
\phi(\rho)=\left\{ \begin{array}{ll} H(\rho \, \| \,\rho_1\otimes K) +
\displaystyle\sum_{a \in \skriy} \rho_2(a) \, \phi_p
\Big(\frac{\rho_1(a)}{\rho_2(a)}\Big)
 & \mbox{ if $\rho_1\ll \rho_2$,} \\
\infty & \mbox{ otherwise,}
\end{array} \right.
\end{align}
where $g_p(x)=\sup_{\lambda\in\R}\big\{\lambda
x+\log(2-e^{\lambda})\big\}.$ Elementary Calculus shows that
\begin{equation}\label{Equ-main}
\sup_{\lambda\in\R}\big\{\lambda x+\log(2-e^{\lambda})\big\}=x\log
x-(x+1)\log\sfrac{(x+1)}{2}.
\end{equation}
Therefore, writing \eqref{Equ-main}  in \eqref{Equ-geo} and
rearranging terms we obtain the form of the rate function in the
corollary which completes the proof.

\bigskip

%%%%%%%%%%%%%%%%%%%%%%%%%%%%%%%%%%%%%%%%%%%%%%%%

\bigskip

\end{document}